\documentclass[reqno,11pt,letterpaper]{amsart}

\usepackage{lipsum}
\usepackage{amsmath}
\usepackage{amssymb}
\usepackage{amsthm}
\usepackage{mathrsfs}
\usepackage{accents}
\usepackage{calc}
\usepackage{arydshln}
\usepackage{upgreek}
\usepackage{slashed}
\usepackage{xifthen}
\usepackage{graphicx}
\usepackage{longtable}
\usepackage[inline]{enumitem}

\usepackage{xcolor}
\definecolor{winered}{rgb}{0.6,0,0}
\definecolor{lessblue}{rgb}{0,0,0.7}

\usepackage[pdftex,colorlinks=true,linkcolor=winered,citecolor=lessblue,urlcolor=lessblue,breaklinks=true,bookmarksopen=true]{hyperref}

\makeatletter
\newcommand{\myitem}[3]{\item[#2]\def\@currentlabel{#3}\label{#1}}
\makeatother

\usepackage{titletoc}

\makeatletter

\def\@tocline#1#2#3#4#5#6#7{
\begingroup
  \par
    \parindent\z@ \leftskip#3 \relax \advance\leftskip\@tempdima\relax
                  \rightskip\@pnumwidth plus 4em \parfillskip-\@pnumwidth
    \ifcase #1 
       \vskip 0.6em \hskip 0em 
       \or
       \or \hskip 0em 
       \or \hskip 1em 
    \fi%
    #6
    \nobreak\relax{\leavevmode\leaders\hbox{\,.}\hfill}
    \hbox to\@pnumwidth {\@tocpagenum{#7}}
  \par
\endgroup
}

 \def\l@section{\@tocline{0}{0pt}{0pc}{}{}}

\renewcommand{\tocsection}[3]{%
  \indentlabel{\@ifnotempty{#2}{ 
    \ignorespaces\bfseries{#2. #3}}}
  \indentlabel{\@ifempty{#2}{\ignorespaces\bfseries{#3}}{}} 
    \vspace{1.5pt}}

\renewcommand{\tocsubsection}[3]{%
  \indentlabel{\@ifnotempty{#2}{
    \ignorespaces#2. #3}}
  \indentlabel{\@ifempty{#2}{\ignorespaces #3}{}}
    \vspace{1.5pt}}

\renewcommand{\tocsubsubsection}[3]{%
  \indentlabel{\@ifnotempty{#2}{
    \ignorespaces#2. #3}}
  \indentlabel{\@ifempty{#2}{\ignorespaces #3}{}}
    \vspace{1.5pt}}

\makeatother

\makeatletter
\def\@nomenstarted{0}
\newlength{\@nomenoldtabcolsep}

\newcommand{\nomenstart}
  {%
    \def\@nomenstarted{1}%
    \setlength{\@nomenoldtabcolsep}{\tabcolsep}%
    \setlength{\tabcolsep}{3.5pt}%
    \begin{longtable}{p{0.11\textwidth} p{0.86\textwidth}}
  }

\newcommand{\nomenitem}[2]{%
    \ifcase\@nomenstarted%
      \or 
      \or \\ 
    \fi%
    #1\,{\leavevmode\leaders\hbox{\,.}\hfill} & #2%
    \def\@nomenstarted{2}%
  }%
\newcommand{\nomenend}
  {\\%
      \end{longtable}%
      \setlength{\tabcolsep}{\@nomenoldtabcolsep}%
      \def\@nomenstarted{0}%
  }
\makeatother

\makeatletter
\newcommand{\vast}{\bBigg@{4}}
\newcommand{\Vast}{\bBigg@{5}}
\newcommand{\VAST}[1]{\bBigg@{#1}}
\makeatother

\allowdisplaybreaks

\numberwithin{equation}{section}
\numberwithin{figure}{section}
\newtheorem{thm}{Theorem}[section]

\newtheorem{prop}[thm]{Proposition}
\newtheorem{lemma}[thm]{Lemma}
\newtheorem{problem}[thm]{Problem}
\newtheorem{cor}[thm]{Corollary}
\newtheorem*{thm*}{Theorem}
\newtheorem*{prop*}{Proposition}
\newtheorem*{cor*}{Corollary}
\newtheorem*{conj*}{Conjecture}

\theoremstyle{definition}
\newtheorem{definition}[thm]{Definition}

\theoremstyle{remark}
\newtheorem{rmk}[thm]{Remark}

\makeatletter
\newcommand{\fakephantomsection}{%
  \Hy@MakeCurrentHref{\@currenvir.\the\Hy@linkcounter}
  \Hy@raisedlink{\hyper@anchorstart{\@currentHref}\hyper@anchorend}%
  \Hy@GlobalStepCount\Hy@linkcounter%
}
\makeatother

\newcommand{\mc}{\mathcal}
\newcommand{\cA}{\mc A}

\newcommand{\cC}{\mc C}

\newcommand{\cE}{\mc E}
\newcommand{\cF}{\mc F}
\newcommand{\cG}{\mc G}

\newcommand{\cK}{\mc K}

\newcommand{\cS}{\mc S}

\newcommand{\cU}{\mc U}

\newcommand{\ms}{\mathscr}

\newcommand{\sS}{\ms S}

\newcommand{\C}{\mathbb{C}}
\newcommand{\N}{\mathbb{N}}
\newcommand{\R}{\mathbb{R}}
\newcommand{\Z}{\mathbb{Z}}

\newcommand{\Sph}{\mathbb{S}}

\newcommand{\fp}{\mathfrak{p}}

\newcommand{\ran}{\operatorname{ran}}
\newcommand{\ann}{\operatorname{ann}}

\newcommand{\coker}{\operatorname{coker}}

\newcommand{\Hom}{\operatorname{Hom}}
\renewcommand{\Re}{\operatorname{Re}}
\renewcommand{\Im}{\operatorname{Im}}

\newcommand{\supp}{\operatorname{supp}}

\newcommand{\tr}{\operatorname{tr}}

\newcommand{\ord}{\operatorname{ord}}
\newcommand{\rank}{\operatorname{rank}}

\newcommand{\diag}{\operatorname{diag}}
\newcommand{\Res}{\operatorname{Res}}

\newcommand{\eps}{\epsilon}
\newcommand{\ftrans}{\;\!\wh{\ }\;\!}

\newcommand{\hra}{\hookrightarrow}
\newcommand{\la}{\langle}

\newcommand{\extcup}{\operatorname{\ol\cup}}

\newcommand{\ol}{\overline}
\newcommand{\pa}{\partial}
\newcommand{\dd}{{\mathrm d}}
\newcommand{\ra}{\rangle}

\newcommand{\specb}{\operatorname{spec}_\bop}
\newcommand{\injspecb}{\operatorname{inj-spec}_\bop}
\newcommand{\surjspecb}{\operatorname{surj-spec}_\bop}
\newcommand{\Specb}{\operatorname{Spec}_\bop}
\newcommand{\injSpecb}{\operatorname{inj-Spec}_\bop}
\newcommand{\surjSpecb}{\operatorname{surj-Spec}_\bop}

\newcommand{\ul}[1]{\underline{#1}{}}

\newcommand{\wh}{\widehat}
\newcommand{\wt}{\widetilde}

\newcommand{\ubar}[1]{\underaccent{\bar}#1}

\newcommand{\bop}{{\mathrm{b}}}

\newcommand{\scop}{{\mathrm{sc}}}

\newcommand{\cp}{{\mathrm{c}}}

\newcommand{\Diff}{\mathrm{Diff}}

\newcommand{\Diffb}{\Diff_\bop}

\newcommand{\Tb}{{}^{\bop}T}

\newcommand{\Tsc}{{}^\scop T}

\newcommand{\Sb}{{}^{\bop}S}

\newcommand{\sigmab}{{}^\bop\upsigma}

\newcommand{\loc}{{\mathrm{loc}}}
\newcommand{\CI}{\cC^\infty}
\newcommand{\CIdot}{\dot\cC^\infty}

\newcommand{\CIc}{\cC^\infty_\cp}

\newcommand{\CmI}{\cC^{-\infty}}

\newcommand{\Hb}{H_{\bop}}

\newcommand{\phg}{{\mathrm{phg}}}

\newcommand{\openbigpmatrix}[1]
  {%
    \def\@bigpmatrixsize{#1}%
    \addtolength{\arraycolsep}{-#1}%
    \begin{pmatrix}%
  }
\newcommand{\closebigpmatrix}
  {%
    \end{pmatrix}%
    \addtolength{\arraycolsep}{\@bigpmatrixsize}%
  }

\newlength{\enummargin}

\newcommand{\usref}[1]{{\upshape\ref{#1}}}

\DeclareGraphicsExtensions{.mps}

\makeatletter
\newcommand*{\fwbw}[1]{\expandafter\@fwbw\csname c@#1\endcsname}
\newcommand*{\@fwbw}[1]{\ifcase #1 \or {\rm fw}\or {\rm bw}\fi}
\AddEnumerateCounter{\fwbw}{\@fwbw}
\makeatother

\begin{document}

\title[Underdetermined-elliptic PDE]{Underdetermined-elliptic PDE on asymptotically Euclidean manifolds, and generalizations}

\date{\today}

\begin{abstract}
  We study underdetermined-elliptic linear partial differential operators $P$ on asymptotically Euclidean manifolds, such as the divergence operator on 1-forms or symmetric 2-tensors. Suitably interpreted, these are instances of (weighted) totally characteristic differential operators on a compact manifold with boundary whose principal symbols are surjective but not injective. We study the equation $P u=f$ when $f$ has a generalized Taylor expansion at $r=\infty$, that is, a full asymptotic expansion into terms with radial dependence $r^{-i z}(\log r)^k$ with $(z,k)\in\C\times\N_0$ up to rapidly decaying remainders. We construct a solution $u$ whose asymptotic behavior at $r=\infty$ is optimal in that the index set of exponents $(z,k)$ arising in its asymptotic expansion is as small as possible. On the flipside, we show that there is an infinite-dimensional nullspace of $P$ consisting of smooth tensors whose expansions at $r=\infty$ contain nonzero terms $r^{-i z}(\log r)^k$ for any desired index set of $(z,k)\in\C\times\N_0$.

  Applications include sharp solvability results for the divergence equation on 1-forms or symmetric 2-tensors on asymptotically Euclidean spaces, as well as a regularity improvement in a gluing construction for the constraint equations in general relativity recently introduced by the author.
\end{abstract}


\author{Peter Hintz}
\address{Department of Mathematics, ETH Z\"urich, R\"amistrasse 101, 8092 Z\"urich, Switzerland}
\email{peter.hintz@math.ethz.ch}

\maketitle


\section{Introduction}
\label{SI}

Underdetermined-elliptic partial differential equations (PDE) arise frequently in geometric analysis, and their solutions are related to deformations of geometric structures; see for example \cite{BourguignonEbinMarsdenPsdoSurj} for some early applications on compact Riemannian manifolds, \cite{CorvinoScalar} for applications to gluing problems for the constraint equations in general relativity, and \cite{DelayCompact} for general results related to (linear) gluing problems. The purpose of the present paper is to study a class of underdetermined PDE on noncompact manifolds and the behavior of its solutions at infinity. The following result serves as an illustration:

\begin{thm}[Divergence on symmetric 2-tensors]
\label{ThmI}
  Let $g$ be the Euclidean metric on $\R^n$, $n\geq 2$. Given a symmetric 2-tensor $h$ on $\R^n$, write $(\delta_g h)_i=-\sum_{j=1}^n\pa_j h_{i j}$ for its (negative) divergence.
  \begin{enumerate}
  \item\label{ItISolv}{\rm (Sharp solvability.)} Let $\omega$ be a 1-form on $\R^n$ with rapid decay, that is, $\omega_i\in\sS(\R^n)$. Then there exists a smooth symmetric 2-tensor $h$ with $\delta_g h=\omega$ and so that in $|x|>1$ we can write $h_{i j}=|x|^{-(n-1)}f_{i j}(|x|^{-1},\frac{x}{|x|})$ for $f_{i j}\in\CI([0,1)\times\Sph^{n-1})$.
  \item\label{ItIker}{\rm (Infinite-dimensional nullspace.)} For any $\alpha<\beta$, there exist infinitely many linearly independent solutions of the homogeneous equation $\delta_g h=0$ so that $h_{i j}\in S^\alpha\setminus S^\beta$ where $S^\alpha(\R^n)$ is the standard symbol class on $\R^n$; that is, $|\pa_x^\gamma h_{i j}(x)|\lesssim\la x\ra^{\alpha-|\gamma|}$ for all $\gamma$, but these bounds fail for $\beta$ in place of $\alpha$. One can find $h$ of this type which moreover have a full generalized Taylor expansion at $|x|^{-1}=0$ with a nontrivial leading order term $\sim|x|^{-i z}(\log|x|)^k$ for any fixed $(z,k)\in\C\times\N_0$.
  \item\label{ItIGen}{\rm (Geometric generalization.)} Suppose that $g$ is an asymptotically Euclidean metric in the sense that $g_{i j}-\delta_{i j}=\la x\ra^{-1}\tilde g_{i j}$ with $\tilde g_{i j}$ smooth in $x$, and for $|x|>1$ also in $|x|^{-1}$, $\frac{x}{|x|}$. Then the same conclusions hold, except now we can typically only find a solution of $\delta_g h=\omega\in\sS$ which has an additional logarithmic subleading term; that is, we can find a smooth solution $h$ of the form $h=h_1+h_2$ where $h_{1,i j}\in|x|^{-n-1}\CI$ and $h_{2,i j}\in|x|^{-n}(\log|x|)\CI$ in $|x|>1$.
  \end{enumerate}
\end{thm}

See Theorem~\ref{ThmAGdiv2} for part~\eqref{ItIGen}, Remark~\ref{RmkAGdivEucl1} for its strengthening in part~\eqref{ItISolv} (which can also be proved directly using the Fourier transform), and Theorems~\ref{ThmAsIndetCon} and \ref{ThmAsIndetPhg} for part~\eqref{ItIker}. Our main result, Theorem~\ref{ThmAs}, which substantially generalizes part~\eqref{ItIGen} in a manner described around~\eqref{EqIEucl} below, also describes the asymptotic behavior of optimal solutions when the right hand side itself has an asymptotic expansion at infinity. Our methods for the construction of solutions with optimal asymptotics apply verbatim also to elliptic equations (such as the Laplacian on tensors), except of course part~\eqref{ItIker} is then no longer valid.

It is a classical result \cite[Theorem~1]{BourguignonEbinMarsdenPsdoSurj} (see also \cite{ChodoshLinStab} for an exposition) that for a differential operator $P\in\Diff^m(U;E,F)$, defined on an open set $U$ and acting between sections of the vector bundles $E$ and $F$, the nullspace of $P$ on $\CIc(U;E)$ is infinite-dimensional if $P$ is \emph{underdetermined-elliptic}, i.e.\ if the principal symbol of $P$ is surjective but not injective everywhere. This applies to the divergence operator in Theorem~\ref{ThmI}. Part~\eqref{ItIker} moreover shows that even if one insists on lower and upper bounds on the decay at infinity, or even on the existence of a nontrivial asymptotic expansion, the nullspace remains infinite-dimensional.

The main focus of the present paper is on showing the existence of solutions of un\-der\-de\-ter\-mined-elliptic PDE $P u=f$ on a manifold $M^\circ$ with optimal asymptotics at infinity, as in parts~\eqref{ItISolv} and \eqref{ItIGen} of Theorem~\ref{ThmI}. The standard approach for studying such PDE is a $P P^*$ argument: one seeks $u$ of the form $P^*v$ where $v$ solves the \emph{elliptic} equation $P P^*v=f$. As an explicit demonstration that this method is inadequate for our aims, consider the divergence equation for 1-forms on $\R^3$,
\begin{equation}
\label{EqIdiv}
  \delta\omega = u \in \sS(\R^3).
\end{equation}
If one were to solve this using a $P P^*$ argument, one would first solve $\delta\delta^*v=\delta\dd v=\Delta v=u$; a solution is given by $v=(4\pi|x|^{-1})*u=Y_0|x|^{-1}+Y_1(\frac{x}{|x|})|x|^{-2}+Y_2(\frac{x}{|x|})|x|^{-3}+\ldots$, where $Y_\ell$, $\ell\in\N_0$, is a (typically nonzero) spherical harmonic of degree $\ell$. Therefore the solution $\omega=\dd v$ of~\eqref{EqIdiv} obtained in this manner satisfies $\omega_i\in|x|^{-2}\CI([0,1)_{|x|^{-1}}\times\Sph^{n-1}_{x/|x|})$ (which happens to be optimal), or $\omega_i\in|x|^{-3}\CI$ when $u$ is orthogonal to constants. In the latter case, however, the optimal solution is Schwartz. (This can be seen by the following elementary argument: since $\int_{\R^3}u(x)\,\dd x=\hat u(0)=0$, one can write $\hat u(\xi)=\sum_{j=1}^3\xi^j\wh{u_j}(\xi)$ for some $\wh{u_j}\in\sS(\R^3)$, and therefore $u=\sum_{j=1}^3 D_{x^j}u_j=\delta_g\omega$ where $\omega_j=i u_j\in\sS(\R^3)$.) As a simple application of our results, one can find a Schwartz solution of~\eqref{EqIdiv} when $u\perp 1$ also on \emph{asymptotically} Euclidean spaces; see Theorem~\ref{ThmAGdiv}.

\begin{rmk}[PDE on closed manifolds]
\label{RmkIClosed}
  If $M$ is a closed manifold (smooth, no boundary) and $P\in\Diff^m(M;E,F)$ is underdetermined-elliptic, then $\ker_{\CI(M)}P$ is infinite-dimensional, and $P(\CI(M))\subset\CI(M)$ is closed, of finite codimension, and (upon fixing densities and fiber inner products) equal to the $L^2$-orthogonal complement of $\ker_{\CI(M)}P^*$. (This can be proved using a $P P^*$ argument \cite{BergerEbinTensors}.) The issue of sharp asymptotics studied in the present paper of course does not arise in this setting.
\end{rmk}

If $w>0$ is a weight function on $M^\circ$, such as $w=\la x\ra^{-\alpha}$ on $\R^n$, one can use the $P P^*$ method for the conjugated operator $w^{-1}P w$, i.e.\ one solves $w^{-1}P w^2 P^* w^{-1}v=w^{-1}f$ and sets $u=w^2 P^* w^{-1}v$. In the case $P=\delta_g$ considered above, this does produce solutions $u$ with any desired amount of decay $\la x\ra^{-\alpha}$ relative to $L^2$ when $u\perp 1$, and indeed with full asymptotic expansions into terms $\sim|x|^{-i z}(\log|x|)^k$ where the exponents $(z,k)\in\C\times\N_0$ depend on the weight $\alpha$; but for no (polynomial) weight does this method produce a solution with optimal (Schwartz) asymptotics.

In situations more complicated than~\eqref{EqIdiv}, the asymptotics produced by a (weighted) $P P^*$ method are weaker even at leading order than the optimal solution, whether or not the right hand side is orthogonal to (part of) the cokernel or not. This happens for example in the study of the linearized constraints map on asymptotically flat initial data sets in general relativity; see the discussion of \cite[Propositions~4.10, 4.14]{HintzGlueID}. Optimizing the asymptotic behavior of solutions of the linearized constraints was the author's original motivation for the present work; see~\S\ref{SsAI}.

\medskip

We shall prove that the optimal asymptotic behavior can essentially be read off from the allowed asymptotic behavior of (approximate) elements of the nullspace of $P^*$: when $P$ is the divergence on 1-forms, the nullspace of $P^*$ consists only of constants; and when $P$ is the divergence on symmetric 2-tensors, then the kernel of $P^*$ consists of the Killing 1-forms on Euclidean space, i.e.\ it is spanned by the generators of translations ($\dd x^j$) and rotations ($x^i\,\dd x^j-x^j\,\dd x^i$). The asymptotic expansions of optimal solutions of $P u=f$ then have (at worst) terms $\sim|x|^{-i z}(\log|x|)^k$ whose behavior at infinity just barely forbids integration by parts against elements of this approximate cokernel; this is the origin of the $|x|^{-(n-1)}$ asymptotics in Theorem~\ref{ThmI}\eqref{ItISolv}.

The proof strategy is thus to first find a formal solution of $P u=f$ near infinity, i.e.\ $u_0$ with an optimal asymptotic expansion so that $f-P u_0$ is Schwartz. (See Proposition~\ref{PropAs}.) Solving away this Schwartz error then first requires adding to $u_0$ a suitable tensor $u_1$ (with optimal asymptotics) so that $f-P u_0-P u_1$ is orthogonal to the tempered distributional kernel of $P^*$. (See Proposition~\ref{PropAsSchw}.) Solving away the remaining Schwartz error is effected using the mapping properties of $P$ on Schwartz spaces (Corollary~\ref{CorbSchwartz}) which is obtained as a simple application of a well-known functional analytic result (reproduced in Appendix~\S\ref{SCl}) but which may nonetheless be of independent interest.

\medskip

The general setting into which we embed the analysis of geometric operators on asymptotically Euclidean spaces is that of \emph{totally characteristic operators} or \emph{b-differential operators} in the parlance of \cite{MelroseMendozaB,MelroseAPS}. These can be regarded as differential operators $P$ on noncompact manifolds $M^\circ$ with a specific structure near infinity: if $M^\circ$ is the interior of a compact manifold $M$ with boundary (the `boundary at infinity'), then the building blocks of b-differential operators are the vector fields on $M$ which are tangent to $\pa M$. See~\S\ref{Sb} for details. In the analysis of elliptic b-differential operators $P$, the \emph{boundary spectrum} $\Specb(P)\subset\C\times\N_0$ plays a key role in determining the asymptotic behavior of solutions. By contrast, in the underdetermined-elliptic case, we need to introduce the finer \emph{surjective boundary spectrum} $\surjSpecb(P)$ for this purpose (which is typically much smaller than the boundary spectrum of $P P^*$ or of weighted versions thereof); this captures the allowed asymptotics of elements of the (approximate) nullspace of $P^*$. The proofs of our main results require substantial revisions of various arguments especially from \cite[\S6]{MelroseAPS} so as to handle the asymmetry between $P$ and $P^*$ in underdetermined-elliptic settings.

Theorem~\ref{ThmI} is an instance of this general b-perspective in the following way. The relevant manifold with boundary is the \emph{radial compactification} of $\R^n$, defined as
\begin{equation}
\label{EqIEucl}
  \ol{\R^n} := \bigl(\R^n \sqcup [0,\infty)_\rho\times\Sph^{n-1}_\omega \bigr) / \sim,\qquad 0\neq x=r\omega\sim(\rho,\omega)=(r^{-1},\omega),
\end{equation}
where $r=|x|$, $\omega=\frac{x}{|x|}$; the boundary is the sphere at infinity $\rho^{-1}(0)$. The function $\la x\ra^{-1}=(1+|x|^2)^{-1/2}$ is then smooth on $\ol{\R^n}$ and vanishes simply at the boundary. Expressing tensors on $\R^n$ in the standard frames $\pa_{x^1},\ldots,\pa_{x^n}$, $\dd x^1,\ldots,\dd x^n$, one then finds that $P:=\la x\ra\delta_g$ is a matrix of b-differential operators on $\ol{\R^n}$ (with smooth coefficients) when $g$ is as in Theorem~\ref{ThmI}. Upon computing $\surjSpecb(P)$, one can then deduce Theorem~\ref{ThmI} from Theorem~\ref{ThmAs}. The details are given in~\S\ref{SsAG}.

\medskip

The plan of the paper is as follows. In~\S\ref{Sb}, we recall elements of the analysis of totally characteristic differential operators and introduce some novel notions tailored to underdetermined-elliptic settings. In~\S\ref{SAs}, we prove our main results (Theorems~\ref{ThmAs}, \ref{ThmAsIndetCon}, and \ref{ThmAsIndetPhg}). The aforementioned applications are discussed in~\S\ref{SA}. We have made an effort to keep the paper self-contained (with the exception of~\S\ref{SsAI} and the usage of basic results from elliptic PDE theory on $\R^n$). In particular, familiarity with b-analysis is \emph{not} assumed.

\subsection*{Acknowledgments}

I would like to thank Hans Lindblad for letting me peek into \cite{HormanderFA} from where I first learned about Theorem~\ref{ThmCl}, and Jure Kali\v{s}nik for discussions regarding Problem~\ref{ProbARI}. Many thanks are also due to a referee for many helpful comments and suggestions.

\section{Totally characteristic differential operators}
\label{Sb}

The material in this section is largely standard (see e.g.\ \cite{MelroseAPS,GrieserBasics}), except for the second part of Definition~\ref{DefbSpecb} and the part starting with Corollary~\ref{CorbSchwartz} and ending with Proposition~\ref{PropACI}. The discussion will thus be relatively brief, but we indicate most proofs in order to make the paper self-contained.

Let $M$ be a compact manifold of dimension $n\in\N$ with non-empty boundary $\pa M$. We assume that $\pa M$ is an embedded submanifold, which ensures the existence of a product neighborhood of $\pa M$ in $M$ (see~\eqref{EqbCollar}). For simplicity of exposition, we shall assume that $\pa M$ is connected (and thus necessarily $n\geq 2$); we leave the minor (largely notational) modifications required to handle the case of disconnected $\pa M$ to the interested reader. We write $M^\circ=M\setminus\pa M$ for the manifold interior of $M$, and we identify
\begin{equation}
\label{EqbCollar}
  [0,1)_\rho \times \pa M
\end{equation}
with a collar neighborhood $\cU\subset M$ of $\pa M$. Without loss of generality, we may assume that $\rho$ is the restriction to $\cU$ of a smooth function on $M$ which vanishes only at $\pa M$. Local coordinates on $\pa M$ will be denoted $y\in\R^{n-1}$. Denote by $E,F\to M$ two vector bundles and fix bundle isomorphisms $E|_\cU\cong\pi^*(E|_{\pa M})$, $F|_\cU\cong\pi^*(F|_{\pa M})$, where $\pi\colon\cU=[0,1)\times\pa M\to\pa M$ is the projection to the second factor. We moreover assume that $E,F$ come equipped with non-degenerate Hermitian fiber inner products. We fix a smooth positive density $\mu$ on $M^\circ$ with the property that there exists
\[
  w\in\R
\]
so that $\rho^{-w+1}\mu$ is a smooth positive density on $\cU$; equivalently, $\mu=\rho^w|\frac{\dd\rho}{\rho}\nu|$ where $0<\nu\in\CI([0,1);\CI(\pa M;\Omega\pa M))$ is a density on $\pa M$ which depends smoothly on $\rho$. We call $w$ the \emph{weight} of this density.\footnote{In asymptotically Euclidean settings in $n$ dimensions, the metric density is of this form for $w=-n$; see~\S\ref{SsAG}.} Adjoints of operators on $M$ acting between spaces of sections of $E$ and $F$ are defined using these fiber inner products and the density $\mu$.

\begin{definition}[Totally characteristic operators]
\label{Defb}
  Let $m\in\N_0$. An $m$-th order \emph{totally characteristic differential operator} (or \emph{b-differential operator}) $P$ on $M$ acting between sections of $E$ and $F$ is an $m$-th order differential operator on $M^\circ$, acting between sections of $E|_{M^\circ}$ and $F|_{M^\circ}$, with smooth coefficients (i.e.\ $P\in\Diff^m(M^\circ;E,F)$) so that in $\cU$ we can write
  \[
    P = \sum_{j=0}^m P_j(\rho)(\rho D_\rho)^j,\qquad P_j\in\CI\bigl([0,1)_\rho;\Diff^{m-j}(\pa M;E|_{\pa M},F|_{\pa M})\bigr),
  \]
  where $D=i^{-1}\pa$. The space of such operators is denoted $\Diffb^m(M;E,F)$. The \emph{normal operator} of such a $P\in\Diffb^m(M;E,F)$ is defined as
  \[
    N(P) := \sum_{j=0}^m P_j(0)(\rho D_\rho)^j \in \Diffb^m([0,\infty)_\rho\times\pa M;\pi^*E|_{\pa M},\pi^*F|_{\pa M}),
  \]
  and the \emph{Mellin-transformed normal operator family} is
  \[
    N(P,z) := \sum_{j=0}^m P_j(0)z^j \in \Diff^m(\pa M;E|_{\pa M},F|_{\pa M}),\qquad z\in\C.
  \]
\end{definition}

Note that $N(P)$ is invariant under dilations $(\rho,y)\mapsto(a\rho,y)$, $a>0$. Moreover, $P-N(P)\in\rho\Diffb^m$ on $\cU$. Thus, elements $u\in\ker N(P,z)$ give rise to elements $\rho^{i z}u$ in the kernel of $N(P)$, and thus (upon cutting them off to a neighborhood of $\pa M$) in the approximate kernel of $P$.

\begin{definition}[Generalized boundary data]
\label{DefbSpecb}
  Let $P\in\Diffb^m(M;E,F)$. The \emph{generalized boundary data} of $P$ are the spaces
    \[
      F(P,z) := \left\{ u = \sum_{k=0}^j \rho^{i z}(\log\rho)^k u_k \colon j\in\N_0,\ u_k\in\CI(\pa M;E|_{\pa M}),\ N(P)u=0 \right\}
    \]
    for $z\in\C$. If $j$ is the largest possible exponent of $\log\rho$ with nonzero coefficient $u_j$ among all elements $u\in F(P,z)$, then we define the \emph{order} of $z$ to be $\ord(P,z):=j+1$; and the \emph{rank} of $z$ is $\rank(P,z):=\dim F(P,z)$. (If $F(P,z)=\{0\}$, we set $\ord(P,z)=0$.) If $\ord(P,z)<\infty$, resp.\ $\ord(P^*,\bar z+i w)<\infty$ for all $z\in\C$, we define the \emph{injective}, resp.\ the \emph{surjective boundary spectrum} of $P$ by\footnote{The surjective boundary spectrum is independent of the choice of volume density: if one passes to a volume density with a different weight $w'$, then adjoints of differential operators are changed via conjugation by $\rho^{w-w'}$ accounting for the difference in weights---which is exactly balanced by the change of the argument of $\ord(P^*,\cdots)$ from $\bar z+i w$ to $\bar z+i w'$.}
    \begin{align*}
      \injSpecb(P) &= \{ (i z,k) \in \C\times\N_0 \colon F(P,z)\neq 0,\ k=\ord(P,z)-1 \}, \\
      \surjSpecb(P) &= \{ (i z,k) \in \C\times\N_0 \colon F(P^*,\bar z+i w)\neq 0,\ k=\ord(P^*,\bar z+i w)-1 \}.
    \end{align*}
    We set $\injspecb(P)=\pi_1(\injSpecb(P))$ and $\surjspecb(P)=\pi_1(\surjSpecb(P))$ where $\pi_1\colon\C\times\N_0\to\C$ is the projection.
\end{definition}

In local coordinates $y\in\R^{n-1}$ on $\pa M$, we can write
\[
  P = \sum_{j=0}^m \sum_{|\alpha|\leq m-j} p_{j\alpha}(\rho,y)(\rho D_\rho)^j D_y^\alpha
\]
for coefficients $p_{j\alpha}$ which are smooth bundle maps from $E$ to $F$. The \emph{b-principal symbol} of $P$ in these coordinates is then the homogeneous polynomial
\begin{equation}
\label{EqbEllSymb}
  \sigmab^m(P)(\rho,y;\xi,\eta) := \sum_{j+|\alpha|=m} p_{j\alpha}(\rho,y)\xi^j\eta^\alpha,\qquad (0,0)\neq (\xi,\eta)\in\R\times\R^{n-1},
\end{equation}
with values in $\Hom(E_{(\rho,y)},F_{(\rho,y)})$. In local coordinates $x\in\R^n$ in $M^\circ$, the b-principal symbol of $P=\sum_{|\beta|\leq m} p_\beta(x)D_x^\beta$ is defined like the ordinary principal symbol, so $\sigmab^m(P)(x;\xi)=\sigma^m(P)(x;\xi)=\sum_{|\beta|=m} p_\beta(x)\xi^\beta\in\Hom(E_x,F_x)$ for $0\neq\xi\in\R^n$. We say that $P$ is
\begin{enumerate}
\item \emph{left elliptic} if its principal symbol is injective, i.e.\ $\sigmab^m(P)(x;\xi)\colon E_x\to F_x$, $\xi\neq 0$, and $\sigmab^m(P)(\rho,y;\xi,\eta)\colon E_{(\rho,y)}\to F_{(\rho,y)}$, $(\xi,\eta)\neq(0,0)$, are injective for all $x,\xi$ and $\rho,y,\xi,\eta$;
\item \emph{right elliptic} if its principal symbol is surjective;
\item \emph{elliptic} if its principal symbol is invertible.
\end{enumerate}
If $P$ is left but not right elliptic, one calls $P$ \emph{overdetermined-elliptic}, and if $P$ is right but not left elliptic, one calls $P$ \emph{underdetermined-elliptic}. Note that if $P$ is underdetermined-elliptic, then necessarily $\rank E>\rank F$.

\begin{lemma}[Consequences of ellipticity]
\label{LemmabEll}
  Let $P\in\Diffb^m(M;E,F)$.
  \begin{enumerate}
  \item\label{ItbEllInj} If $P$ has injective principal symbol, then $\injSpecb(P)$ is well-defined (i.e.\ $ord(P,z)<\infty$ for all $z\in\C$) and discrete, and its intersection with a strip $\{z\in\C\colon |\Im z|<C\}$ is a finite set for all $C$.
  \item\label{ItbEllSurj} If $P$ has surjective principal symbol, then the conclusions of part~\eqref{ItbEllInj} hold for $\surjSpecb(P)$ in place of $\injSpecb(P)$.
  \item\label{ItbEll} $P$ is elliptic, i.e.\ has invertible principal symbol, then we have $\injSpecb(P)=\surjSpecb(P)=:\Specb(P)$.
  \end{enumerate}
\end{lemma}
\begin{proof}
  Part~\eqref{ItbEllInj} follows from part~\eqref{ItbEll} by considering the elliptic operator $P^*P$; note that $\ker P\subset\ker P^*P$, and therefore $F(P,z)\subset F(P^*P,z)$. Similarly, part~\eqref{ItbEllSurj} follows by considering the elliptic operator $P P^*$. Part~\eqref{ItbEll} finally is standard, see e.g.\ \cite[\S\S5.2--5.3]{MelroseAPS}; in brief, the ellipticity of $P$ implies that for $C>0$, the operator family $N(P,z)\in\Diff^m(\pa M;E|_{\pa M},F|_{\pa M})$ is elliptic with large parameter $z\in\C$ in the strip $|\Im z|<C$, and therefore it is invertible for $|\Re z|>C'$ with $C'$ depending on $C$; see e.g.\ \cite[Chapter II.9]{ShubinSpectralTheory}. The inverse $N(P,z)^{-1}$ is finite-meromorphic.

  For the equality $\injSpecb(P)=\surjSpecb(P)$, note that $N(P,z)^*=(\rho^{-i z}N(P)\rho^{i z})^*=\rho^{-i\bar z}N(P)^*\rho^{i\bar z}$, where the adjoint on the right is taken with respect to a dilation-invariant (thus weight $0$) density $[0,\infty)\times\pa M$ which equals $\rho^{-w}\mu$ to leading order at $\rho=0$. Passing to the weight $w$ density $\mu$ on $M$, this is equal to $\rho^{-i\bar z}\rho^w N(P^*)\rho^{-w}\rho^{i\bar z}=N(P^*,\bar z+i w)$. The claim then follows from the characterization
  \[
    F(P,z_0) = \left\{ \Res_{z=z_0} N(P,z)^{-1}f(z) \colon f\ \text{is a polynomial with values in}\ \CI(\pa M;F|_{\pa M}) \right\},
  \]
  which implies that $\ord(P,z_0)$ is the order of the pole of (the finite-meromorphic operator family) $N(P,z)^{-1}=(N(P^*,\bar z+i w)^{-1})^*$ at $z=z_0$.
\end{proof}

If $P\in\Diffb^m(M;E,F)$ has injective principal symbol, then for all $s,C\in\R$ there exist $C_s,C'$ so that
\begin{equation}
\label{EqbEllEst}
  \|u\|_{H_{\la z\ra^{-1}}^s(\pa M;E|_{\pa M})} \leq C_s\la z\ra^{-m}\|N(P,z)u\|_{H_{\la z\ra^{-1}}^{s-m}(\pa M;E|_{\pa M})},\qquad |\Im z|<C,\ |\Re z|>C'.
\end{equation}
The function space is defined via a partition of unity by means of the $\R^n$-version $H_h^s(\R^n)$, $h>0$, which is defined to be $H^s(\R^n)$ but with norm $\|u\|_{H_h^s(\R^n)}:=\|\la h\xi\ra^s\hat u(\xi)\|_{L^2_\xi}$; here, $\hat u(\xi)$ denotes the Fourier transform of $u$. Indeed, the estimate~\eqref{EqbEllEst} follows for elliptic $P$ from the ellipticity with large parameter of $N(P,z)$, which implies that $h^m N(P,z)$, with $h=\la z\ra^{-1}$, is an elliptic semiclassical operator on $\pa M$; see \cite[Theorem~E.33]{DyatlovZworskiBook} (applied with $A=B_1=\chi=1$ and $N(P,z)$ in place of $P$). If $\alpha\in\R$ is such that $N(P,\lambda)$ is invertible for all $\lambda\in\C$ with $\Im\lambda=-\alpha$, then the estimate~\eqref{EqbEllEst} holds for all $\Im\lambda=-\alpha$: for bounded $\lambda$, this estimate is a consequence of the ellipticity and injectivity of $N(P,\lambda)$. If $P$ merely has injective principal symbol, then~\eqref{EqbEllEst} follows by applying what we have already shown to the elliptic operator $P^*P$.

We study b-differential operators themselves on b-Sobolev spaces $\Hb^{s,\alpha}(M,\mu)$. Recall that $\mu$ is a smooth positive density on $M^\circ$ which near $\pa M$ is of the form $\mu=\rho^w|\frac{\dd\rho}{\rho}\nu|$ where $0<\nu\in\CI([0,1);\CI(\pa M;\Omega\pa M))$. The convention for the weight $\alpha$ in the following definition ensures that $\Hb^{0,0}(M,\mu)=L^2(M,\mu)$; the reader may wish to consider the case $w=0$ (i.e.\ $\mu$ is an \emph{unweighted b-density}) first.

\begin{definition}[Function spaces]
\fakephantomsection
\label{DefbFn}
  \begin{enumerate}
  \item We write $\CIdot(M)\subset\CI(M)$ for the Schwartz space on $M$, i.e.\ the space of all smooth functions vanishing to infinite order at $\pa M$.
  \item Elements of the dual space $\CmI(M)=\CIdot(M)^*$ are \emph{tempered distributions} on $M$.
  \item Let $\chi\in\CIc([0,1)\times\pa M)$ be equal to $1$ near $\{0\}\times\pa M$. Define for $s\in\R$ the space $\Hb^{s,w/2}(M,\mu)$ as the space of all $u\in\CmI(M)$ so that $(1-\chi)u\in H^s_\loc(M^\circ)$, and so that all localizations of $\chi u$ to products of $[0,1)_\rho$ with local coordinate charts on $\pa M$ using a partition of unity lie in $\Hb^s([0,\infty)_\rho\times\R^{n-1})$ which we define to be the pullback of $H^s(\R\times\R^{n-1})$ under the map
  \begin{equation}
  \label{EqbMap}
    (\rho,y)\mapsto(x,y):=(-\log\rho,y).
  \end{equation}
  \item We further let $\Hb^{s,\alpha}(M,\mu)=\rho^{\alpha-w/2}\Hb^{s,w/2}(M,\mu)=\{\rho^{\alpha-w/2}u\colon u\in\Hb^{s,w/2}(M,\mu)\}$ for $\alpha\in\R$. Spaces $\Hb^{s,\alpha}(M,\mu;E)$ of sections of $E$ are defined via a partition of unity and local trivializations of $E$. Finally, we write $\Hb^s(M,\mu;E)=\Hb^{s,0}(M,\mu;E)$.
  \end{enumerate}
\end{definition}

This definition ensures that $\Hb^0(M,\mu)=L^2(M,\mu)$; note indeed that the preimage of $|\dd x\,\dd y|$ is $|\frac{\dd\rho}{\rho}\dd y|$, so the two spaces agree when $\mu$ has weight $0$, and the general case follows from this. For $k\in\N_0$, the space $\Hb^k(M,\mu)$ consists of all $u\in\Hb^0(M,\mu)$ so that $P u\in\Hb^0(M,\mu)$ for all $P\in\Diffb^m(M)$. Furthermore, since $D_x=-\rho D_\rho$, and since the coefficients of $P\in\Diffb^m(M)$ in local coordinates near $\pa M$ push forward to uniformly bounded (with all derivatives) smooth functions on $\R\times\R^{n-1}$, we have $P\colon\Hb^{s,\alpha}(M,\mu;E)\to\Hb^{s-m,\alpha}(M,\mu;F)$ for $P\in\Diffb^m(M;E,F)$. One can endow $\Hb^{s,w/2}(M,\mu)$ with the structure of a Hilbert space by taking as a squared norm the sum of the squares of $H^s$-norms in a finite system of local coordinates on $\supp(1-\chi)\subset M^\circ$ plus the squared $H^s$-norms of the pushforwards to $\R\times\R^{n-1}$ of the localizations to coordinate charts near the boundary. This induces a Hilbert space structure on $\Hb^{s,\alpha}(M,\mu)$ for all $\alpha\in\R$. The $L^2(M,\mu)$-inner product induces an isomorphism between $\Hb^{-s,-\alpha}(M,\mu)$ and the dual of $\Hb^{s,\alpha}(M,\mu)$.

Finally, using the Mellin transform $\hat u(z,y)=\int_0^\infty \rho^{-i z}u(\rho,y)\,\frac{\dd\rho}{\rho}$, we have
\begin{equation}
\label{EqbMT}
  \|\chi u\|_{\Hb^{s,\alpha}(M,\mu)}^2 \sim \int_{\Im z=-(\alpha-\frac{w}{2})} \la z\ra^{2 s} \|\wh{\chi u}(z,-)\|_{H_{\la z\ra^{-1}}^s(\pa M)}^2\,\dd z,
\end{equation}
i.e.\ the left and right hand sides are bounded by a $u$-independent multiple of each other. This follows from the analogous statement
\[
  \|u\|_{H^s(\R\times\R^{n-1})}^2 = \int_\R \la z\ra^{2 s} \|\hat u(z,-)\|_{H^s_{\la z\ra^{-1}}(\R^{n-1})}^2\,\dd z
\]
where $\hat u(z,y)=\int_\R e^{i x z}u(x,y)\,\dd x$ is the Fourier transform (up to a sign) of $u$ in the first coordinate; this in turn uses that $\la z\ra^{2 s}\la \la z\ra^{-1}\eta \ra^{2 s}=(1+|z|^2+|\eta|^2)^s$.

\begin{prop}[(Semi-)Fredholm property]
\label{PropbFred}
  Let $P\in\Diffb^m(M;E,F)$ and $s,\alpha\in\R$; consider $P$ as a bounded linear map
  \begin{equation}
  \label{EqbFred}
    P \colon \Hb^{s,\alpha}(M,\mu;E) \to \Hb^{s-m,\alpha}(M,\mu;F).
  \end{equation}
  \begin{enumerate}
  \item\label{ItbFredInj} If $P$ has injective principal symbol and $\alpha-\frac{w}{2}\notin\Re\injspecb(P)$, then~\eqref{EqbFred} has finite-dimensional kernel and closed range.
  \item\label{ItbFredSurj} If $P$ has surjective principal symbol and $\alpha-\frac{w}{2}\notin\Re\surjspecb(P)$, then the range of~\eqref{EqbFred} is closed and has finite-dimensional codimension.
  \item\label{ItbFredEll} If $P$ is elliptic, and $\alpha-\frac{w}{2}\notin\Re\specb(P)$, then~\eqref{EqbFred} is Fredholm.
  \end{enumerate}
  In each of these cases, the range of~\eqref{EqbFred} is equal to the $L^2(M,\mu;F)$-annihilator of the kernel of $P^*$ on $\Hb^{-s+m,-\alpha}(M,\mu;F)$.
\end{prop}
\begin{proof}
  Define by $\mu_0=\rho^{-w}\mu$ a density with weight $0$; then $\Hb^{s,\alpha}(M,\mu)=\Hb^{s,\alpha_0}(M,\mu_0)$ where $\alpha_0=\alpha-\frac{w}{2}$. Moreover, if $P_0^*$ denotes the adjoint of $P$ with respect to $\mu_0$, then $P^*=\rho^{-w}P_0^*\rho^w$. We can thus reduce the Proposition to the case that $\mu$ has weight $0$. Furthermore, replacing $P$ by $\rho^{-\alpha}P\rho^\alpha$, we may assume $\alpha=0$. The assumption in part~\eqref{ItbFredInj} is now that $0\notin\Re\injspecb(P)$, similarly for the other parts.

  Parts~\eqref{ItbFredInj} and~\eqref{ItbFredEll} are standard, see e.g.\ \cite[Theorem~5.40]{MelroseAPS}. In brief, elliptic estimates (for $P$ if $P$ is elliptic, and for $P^*P$ when $P$ merely has injective principal symbol) in $M^\circ$ as well as elliptic estimates for the uniformly elliptic operators on $\R\times\R^{n-1}$ arising via pushforward along~\eqref{EqbMap} of localizations of $P$ or $P^*P$ to coordinate charts near $\pa M$ imply for any $s_0\in\R$ an estimate
  \[
    \|u\|_{\Hb^s(M;E)} \leq C\left(\|P u\|_{\Hb^{s-m}(M;F)} + \|u\|_{\Hb^{s_0}(M;E)}\right).
  \]
  We take $s_0<s-1$. Using~\eqref{EqbMT} and the estimate~\eqref{EqbEllEst}---which by assumption is valid for \emph{all} $z\in\R$---, we can then further estimate $\|u\|_{\Hb^{s_0}(M;E)}\leq\|(1-\chi)u\|_{\Hb^{s_0,-1}(M;E)}+\|\chi u\|_{\Hb^{s_0}(M;E)}$ and then (with $C$ changing from line to line)
  \begin{align*}
    \|\chi u\|_{\Hb^{s_0}(M;E)}^2 &\leq C\int_\R \la z\ra^{2 s_0}\|\wh{\chi u}(z,-)\|_{H_{\la z\ra^{-1}}^{s_0}(\pa M;E|_{\pa M})}^2\,\dd z \\
      &\leq C\int_\R \la z\ra^{2(s_0-m)}\|N(P,z)\wh{\chi u}(z,-)\|_{H_{\la z\ra^{-1}}^{s_0-m}(\pa M;F|_{\pa M})}^2\,\dd z \\
      &\leq C\|N(P)(\chi u)\|_{\Hb^{s_0-m}(M;F)}.
  \end{align*}
  Replacing $N(P)$ by $P$ and commuting $P$ through $\chi$, we arrive at
  \[
    \|u\|_{\Hb^s(M;E)} \leq C\left(\|P u\|_{\Hb^{s-m}(M;F)} + \|u\|_{\Hb^{s-1,-1}(M;E)}\right).
  \]
  Since by the Rellich--Kondrakhov theorem the inclusion $\Hb^s(M;E)\hra\Hb^{s-1,-1}(M;E)$ is compact, this implies that $P$ has finite-dimensional kernel and closed range; this proves part~\eqref{ItbFredInj}. When $P$ is elliptic, an analogous estimate for $P^*$ implies also the finite-co\-di\-men\-sio\-na\-li\-ty of the range of $P$.

  For part~\eqref{ItbFredSurj}, we consider the elliptic operator $P P^*$; note then that the range of~\eqref{EqbFred}, which contains the closed and finite-codimensional space $P P^*(\Hb^{s+m}(M,\mu_0;F))$, is itself closed.
\end{proof}

\begin{cor}[Generalized right inverse]
\label{CorbRI}
  Suppose $P\in\Diffb^m(M;E,F)$ has surjective principal symbol. Let $\alpha\in\R$ be such that $\alpha-\frac{w}{2}\notin\Re\surjspecb(P)$. Then there exists an operator $G$ which is continuous as map $G\colon\Hb^{s-m,\alpha}(M,\mu;F)\to\Hb^{s,\alpha}(M,\mu;E)$ for all $s\in\R$ and which is a generalized right inverse of $P$ in the sense that $P G f=f$ for all $f\in\Hb^{s-m,\alpha}(M,\mu;F)$ which are orthogonal to (the finite-dimensional space) $\ker_{\Hb^{\infty,-\alpha}(M,\mu;F)}P^*$ (equivalently: for all $f$ in the range of $P$ on $\Hb^{s,\alpha}(M,\mu;E)$).
\end{cor}
\begin{proof}
  We may reduce to the case $w=0$. We then formally solve the equation $P u=f$ via
  \begin{equation}
  \label{EqbRIT}
    u=\rho^{2\alpha}P^*\rho^{-\alpha}v,\qquad T_\alpha v = \rho^{-\alpha}f, \quad
    T_\alpha := (\rho^{-\alpha}P\rho^\alpha)(\rho^{-\alpha}P\rho^\alpha)^*.
  \end{equation}
  Note then that
  \[
    N(T_\alpha,z)=N(P,z-i\alpha)N(P^*,z+i\alpha) = N(P^*,\bar z+i\alpha)^*N(P^*,z+i\alpha)
  \]
  is invertible for $z\in\R$, and therefore $T_\alpha\colon\Hb^{s,0}(M;F)\to\Hb^{s-2 m,0}(M;F)$ is Fredholm for all $s\in\R$. The $L^2(M;F)$-orthogonal complement of its range (as well as its kernel) is $\rho^\alpha\ker_{\Hb^{0,-\alpha}(M;F)}P^*=\rho^\alpha\ker_{\Hb^{\infty,-\alpha}(M;F)}P^*$ (using elliptic regularity).
\end{proof}

\begin{cor}[Mapping properties on Schwartz spaces]
\label{CorbSchwartz}
  Let $P\in\Diffb^m(M;E,F)$, and consider $P,P^*$ as bounded linear maps
  \[
    P \colon \CIdot(M;E) \to \CIdot(M;F),\qquad
    P^* \colon \CmI(M;F) \to \CmI(M;E).
  \]
  Suppose that the principal symbol of $P$ is injective, surjective, or invertible. Then $P$ has closed range, and $P(\CIdot(M;E))$ is the annihilator of $\ker_{\CmI(M;F)}P^*$.
\end{cor}
\begin{proof}
  We first consider the case that $P$ has surjective principal symbol. Fix sequences $s_j\to\infty$, $\alpha_j\to\infty$ to that $\alpha_j-\frac{w}{2}\notin\Re\surjspecb(P)$ for all $j$. By Sobolev embedding, we have $\CIdot(M;E)=\bigcap_j \Hb^{s_j,\alpha_j}(M,\mu;E)$. In the notation of Theorem~\ref{ThmCl}, we need to show that
  \[
    P^*(\CmI(M;F))\cap U_j\subset\CmI(M;E)
  \]
  is weak* closed, where
  \[
    U_j=\left\{u^*\in\CmI(M;E)\colon|\la u^*,u\ra|\leq\|u\|_{\Hb^{s_j,\alpha_j}(M,\mu;E)}\ \forall\,u\in\CIdot(M;E)\right\}
  \]
  is the unit ball in $\Hb^{-s_j,-\alpha_j}(M,\mu;E)$. Since $\CIdot(M;E)$ is separable, the weak* topology on $U_j$ is metrizable. Suppose now that $P^*f_k^*=u_k^*\in U_j$ is a weak* convergent sequence in $\CmI(M;E)$; its limit satisfies $u^*\in U_j$. Then for all $u\in\ker_{\Hb^{s_j,\alpha_j}(M,\mu;E)}P$ we have $\la u_k^*,u\ra=\la P^*f_k^*,u\ra=\la f_k^*,P u\ra=0$ and therefore $\la u^*,u\ra=0$. But this implies that $u^*\in P^*(\Hb^{-s_j+m,-\alpha_j}(M,\mu;F))$ by Proposition~\ref{PropbFred} (applied to $P^*$), finishing the proof that $P$ has closed range. The case that $P$ has injective principal symbol is handled in a completely analogous manner.
\end{proof}

\begin{problem}[Right inverse]
\label{ProbARI}
  When $P\in\Diffb^m(M;E,F)$ is underdetermined-elliptic, does $P\colon\CIdot(M;E)\to\CIdot(M;F)$ have a continuous right inverse on its range, i.e.\ a continuous right inverse $P(\CIdot(M;E))\to\CIdot(M;E)$?
\end{problem}

This is equivalent to the existence of a complementary subspace of $\ker_{\CIdot(M;E)}P$ in $\CIdot(M;E)$. When $P$ is left elliptic, this is clear since $\ker_{\CIdot(M;E)}P$ is finite-dimensional and thus complemented.

\begin{rmk}[Right inverse on closed manifolds]
\label{RmkRightInv}
  A right elliptic differential operator $Q\in\Diff^m(X;G,H)$ on a closed manifold $X$, with $G,H\to X$ smooth vector bundles, always has a continuous right inverse defined on its range. This follows from the existence of an $L^2(X;G)$-orthogonal splitting $\CI(X;G)=\ker_{\CI(X;G)}Q\oplus Q^*(\CI(X;H))$, with both summands closed in $\CI(X;G)$. The proof of the latter is similar to that of \cite[Corollary~4.2]{BergerEbinTensors} which treats the case that $Q$ is left elliptic: the orthogonal splitting $L^2(X;G)=\ker_{L^2(X;G)}Q\oplus Q^*(H^m(X;H))$ implies for $u\in\CI(X;G)$ the splitting $u=u_0+Q^*u_1$ with $u_0\in\ker_{L^2(X;G)}Q$ and $u_1\in H^m(X;H)$; but then $Q u=Q Q^*u_1\in\CI(X;H)$ implies $u_1\in\CI(X;H)$ by elliptic regularity and thus also $u_0\in\CI(X;G)$. Furthermore, $Q^*(\CI(X;H))$ is closed since $Q^*u_k\to f\in\CI(X;H)$ implies that $f$, like all $Q^*u_k$ is orthogonal to $\ker_{L^2(X;G)}Q$ and thus equal to $Q^*u$ for some $u\in H^m(X;H)$, which must then lie in $\CI(X;H)$ by elliptic regularity.
\end{rmk}

The case of underdetermined-elliptic totally characteristic differential operators appears to be more subtle. (For example, the $\Hb^{0,\alpha}(M;E)$-orthogonal splitting $\Hb^{s,\alpha}(M,\mu;E)=\ker_{\Hb^{s,\alpha}(M,\mu;E)}P\oplus \rho^{2\alpha}P^*(\Hb^{s+m,-\alpha}(M,\mu;F))$, for $s\in\R\cup\{\infty\}$ and $\alpha\notin\Re\surjSpecb(P)$, cannot be used in a fashion analogous to Remark~\ref{RmkRightInv}.) Absent an affirmative resolution of Problem~\ref{ProbARI}, we record instead the following result, which is sufficient for most applications:

\begin{prop}[Smooth solvability for finite-dimensional families]
\label{PropACI}
  Let $B$ be a finite-dimensional smooth manifold. Let $P\in\Diffb^m(M;E,F)$ be left or right elliptic. Let $f\in\CI(B;\CIdot(M;F))$, and suppose $f(b,-)\in\ann\ker_{\CmI(M;F)}P^*$ for all $b\in B$. Then there exists $u\in\CI(B;\CIdot(M;E))$ so that $P(u(b,-))=f(b,-)$ for all $b\in B$.
\end{prop}

See \cite[Theorems~52.5 and 52.6]{TrevesTVS} for general results of this type; we give a direct proof in the present setting.

\begin{proof}[Proof of Proposition~\usref{PropACI}]
  When $P$ is left elliptic or elliptic, this follows from the existence of a continuous right inverse on $P(\CIdot(M;E))$. If $P$ is underdetermined-elliptic, we argue as follows. Using a partition of unity on $B$, a Seeley extension argument in case $B$ has boundary, and using the linearity of $P$, it suffices to consider the case that $B$ is the $N$-torus. Denoting by $P^B\in\Diffb^m(B\times M;\pi_2^*E,\pi_2^*F)$ (with $\pi_2\colon B\times M\to M$ the projection) the operator defined by $(P_B u)(b,-)=P(u(b,-))$, we shall show using Theorem~\ref{ThmCl} that $(P^B)^*\colon\CmI(B\times M;\pi_2^*F)\to\CmI(B\times M;\pi_2^*E)$ has weak* closed range. Suppose $(P^B)^*f_k^*=u_k^*\in H^{-s_j}(B;\Hb^{-s_j,-\alpha_j}(M,\mu;E))$ is a weak* convergent sequence with limit $u^*$; we need to show that $u^*$ lies in the range of $(P^B)^*$. Let us employ the Fourier transform in $B$, denoted by a hat, to pass to spaces of polynomially weighted $\ell^2$-sequences, parameterized by a momentum variable $\beta\in\Z^N$, with values in $\Hb^{-s_j,-\alpha_j}(M,\mu;E)$. Necessarily then, $\wh{u_k^*}(\beta,-)$ lies in the $L^2(M,\mu;E)$-orthogonal complement of $\ker_{\Hb^{s_j,\alpha_j}(M,\mu;E)}P$. Since the kernel of $P^*$ on $\Hb^{-s_j+m,-\alpha_j}(M,\mu;F)$ has a complement given by $V:=\rho^{-2\alpha_j}P(\Hb^{-s_j+2 m,\alpha_j}(M,\mu;E))$, there exists a unique $\tilde f_k^*(\beta,-)\in V$, with norm bounded by a $\beta$-independent constant times that of $\wh{u_k^*}(\beta,-)$, so that $P^*\tilde f_k^*(\beta,-)=\wh{u_k^*}(\beta,-)$. Since $\wh{u_k^*}$ is a Cauchy sequence, so is $\tilde f_k^*$, and the limit $\tilde f^*$ is the Fourier transform of $f^*\in H^{-s_j}(B;\Hb^{-s_j+m,-\alpha_j}(M,\mu;F))$ where $(P^B)^*f^*=u^*$.

  Finally, note that if $f^*\in\CmI(B\times M;\pi_2^*F)$ lies in $\ker(P^B)^*$, then $\wh{f^*}(\beta,-)\in\ker P^*$ for all $\beta\in\Z^N$. The Proposition now follows from the fact that $P^B(\CIdot(B\times M;\pi_2^*E))=\ann\ker_{\CmI(B\times M;\pi_2^*F)}(P^B)^*$.
\end{proof}

Our main interest in this paper is in the mapping properties of $P$ on spaces of polyhomogeneous distributions. 

\begin{definition}[Conormality and polyhomogeneity]
\fakephantomsection
\label{DefAs}
  \begin{enumerate}
  \item\label{ItAsCon} For $\alpha\in\R$, we let $\cA^\alpha(M)=\{u\colon A u\in\rho^\alpha L^\infty(M)\ \forall\,A\in\Diffb(M)\}$ be the space of \emph{conormal functions} with weight $\alpha$.
  \item\label{ItAsPhg} An \emph{index set} is a subset $\cE\subset\C\times\N_0$ so that $(z,k)\in\cE$ implies $(z+1,k)\in\cE$ and also $(z,k-1)\in\cE$ when $k\geq 1$, and for all $C\in\R$ the set of $(z,k)\in\cE$ with $\Re z<C$ is finite. Let $\chi\in\CIc([0,1)_\rho\times\pa M)$ be identically $1$ near $\rho=0$. We then define the space $\cA_\phg^\cE(M)$ of \emph{$\cE$-smooth functions} (or \emph{polyhomogeneous conormal functions with index set $\cE$}) to consist of all $u\in\CI(M^\circ)$ for which there exist $u_{(z,k)}\in\CI(\pa M)$, $(z,k)\in\cE$, so that for all $C\in\R$, we have
  \begin{equation}
  \label{EqAs}
    \left(u(\rho,y)-\sum_{(z,k)\in\cE,\ \Re z\leq C}\rho^z(\log\rho)^k u_{(z,k)}(y)\right)\chi\in\cA^C([0,1)\times\pa M).
  \end{equation}
  \end{enumerate}
  (Polyhomogeneous) conormal sections of a smooth vector bundle $E\to M$ are defined analogously, now with $u_{(z,k)}\in\CI(\pa M;E|_{\pa M})$.
\end{definition}

If $\cE$ is an index set and we are given $u_{(z,k)}\in\CI(\pa M)$, $(z,k)\in\cE$, then there exists $u\in\cA_\phg^\cE(M)$ so that~\eqref{EqAs} holds. (This is a variation of Borel's lemma.) More generally, given index sets $\cE_1,\cE_2,\ldots$ with $C_j:=\min_{(z,0)\in\cE_j}\Re z\to\infty$ as $j\to\infty$, and given $u_j\in\cA_\phg^{\cE_j}(M)$, there exists $u\in\cA_\phg^\cE(M)$, $\cE=\bigcup_j\cE_j$, so that $u-\sum_{j=1}^J u_j\in\cA^{C_J-1}(M)$ for all $J$; such a $u$ is unique modulo $\CIdot(M)$, and is called an \emph{asymptotic sum} of the $u_j$.

\begin{lemma}[Polyhomogeneous nullspace]
\label{LemmaAsKer}
  Suppose $P\in\Diffb^m(M;E,F)$ has injective principal symbol. Let $u\in\CmI(M;E)$ and suppose $P u=0$ (or more generally $P u\in\CIdot(M;F)$). Then there exists $\alpha\in\R$ so that $u\in\cA^\alpha(M;E)$. Denote by $\alpha_0\in\R\cup\{+\infty\}$ the supremum of all such $\alpha$; then $u\in\cA_\phg^\cE(M;E)$ where $\cE=\cE_0(P,\alpha_0)\cup\cE_+(P,\alpha_0)$, where\footnote{Note that $\cE_0(P,\alpha_0)$ is not an index set.} $\cE_0(P,\alpha_0)=\{(z,j)\colon \exists\,(z,k)\in\injSpecb(P),\ \Re z=\alpha_0,\ k\geq j\}$, while $\cE_+(P,\alpha_0)$ is an index set with $\Re z>\alpha_0$ for all $(z,k)\in\cE_+(P,\alpha_0)$. (If $\alpha_0=+\infty$, this means $u\in\CIdot(M;E)$.)
\end{lemma}
\begin{proof}
  The first part of the proof follows \cite[Proposition~5.61]{MelroseAPS}. Let $f=P u$. Since $\CmI(M;E)$ is the union of all weighted b-Sobolev spaces, there exist $s,\beta\in\R$ so that $u\in\Hb^{s,\beta}(M,\mu;E)$. But then elliptic regularity for $P^*P u=P^*f$ implies $u\in\Hb^{\infty,\beta}(M,\mu;E)$, which by Sobolev embedding implies $u\in\cA^\alpha(M;E)$ for $\alpha=\beta-\frac{w}{2}$. If $\alpha_0$, as defined in the statement of the Lemma, equals $+\infty$, then $u\in\CIdot(M;E)$ and we are done. Otherwise, $u\in\cA^{\alpha_0-\eps}(M;E)$ for all $\eps>0$.

  We now work in the collar neighborhood $[0,1)_\rho\times\pa M$ of $\pa M$, and let $\chi\in\CIc([0,1)\times\pa M)$ denote a cutoff which is identically $1$ near $\rho=0$. Passing to the Mellin transform in the equation $N(P^*P)(\chi u)=f_1:=\chi P^*f+(N(P^*P)-P^*P)(\chi u)+[P^*P,\chi]u\in\cA^{\alpha_0-\eps+1}([0,1)\times\pa M;F|_{\pa M})$ gives
  \[
    N(P^*P,z)\wh{\chi u}(z) = \wh{f_1}(z),\qquad \Im z>-\alpha_0,
  \]
  where $\wh{f_1}(z)$ is holomorphic in $\Im z>-\alpha_0-1$, takes values in $\CI(\pa M;F|_{\pa M})$, and vanishes rapidly at real infinity. Since $N(P^*P,z)^{-1}$ is meromorphic and satisfies the bounds~\eqref{EqbEllEst}, we conclude that  $\wh{\chi u}(z)$ extends meromorphically to $\Im z>-\alpha_0-1$ as well and vanishes rapidly at real infinity. Shifting the contour in the inverse Mellin transform
  \[
    \chi u(\rho,y) = \frac{1}{2\pi} \int_{\Im z=-\alpha_0+\eps} \rho^{i z}N(P^*P,z)^{-1}\wh{f_1}(z)\,\dd z
  \]
  to $\Im z=-\alpha_0-1+\eps$ where $\eps\in(0,1)$ is such that no poles of $N(P^*P,z)^{-1}$ have $\Im z=-\alpha_0+\eps,-\alpha_0-1+\eps$, we conclude that $\chi u$ is polyhomogeneous modulo a remainder in $\cA^{\alpha_0-1+\eps}(M;E)$. Iterating this argument establishes the polyhomogeneity of $u$.

  To get precise information about the leading order part of the index set of $u$, we return to $P u=f$ and observe that
  \[
    N(P,z)\wh{\chi u}(z) = \wh{f_2}(z),\qquad \Im z>-\alpha_0-1,
  \]
  where $f_2=\chi f+(N(P)-P)(\chi u)+[P,\chi]u\in\cA^{\alpha_0-1+\eps}([0,1)\times\pa M;F|_{\pa M})$. Note that $\wh{f_2}(z)$ is holomorphic for $\Im z>-\alpha_0-1$. Therefore, if $z\in\C$ with $\Im z\in(-\alpha_0-1,-\alpha_0]$ is a pole of $\wh{\chi u}(z)$ of order $j\geq 1$, then $(i z,k-1)\in\injSpecb(P)$ for some $k\geq j$. This completes the proof.
\end{proof}

\section{Solutions of underdetermined-elliptic PDE with sharp asymptotics}
\label{SAs}

We continue using the notation from the previous section. Let $P\in\Diffb^m(M;E,F)$.

\subsection{Formal solutions}
\label{SsAsFormal}

For $j\in\N_0$ and $z_0\in\C$, we define
\begin{align}
\label{EqAsFj}
\begin{split}
  \hat F_j(P,z_0) &:= \Biggl\{ \tilde u(z)=\sum_{k=0}^j (z-z_0)^{-k-1}u_k \colon u_0,\ldots,u_j\in\CI(\pa M;E|_{\pa M}), \\
    &\hspace{12em} N(P,z)\tilde u(z)\ \text{is holomorphic at}\ z=z_0 \Biggr\},
\end{split} \\
  \hat F_{[j]}(P,z_0) &:= \{ {\rm L.o.t.}(\tilde u) \colon \tilde u\in\hat F_j(P,z_0) \} \subset \CI(\pa M;E|_{\pa M}); \nonumber
\end{align}
here we set ${\rm L.o.t.}(\tilde u(z)):=u_j$ for $\tilde u(z)$ as in~\eqref{EqAsFj}.\footnote{Thus, the map $\rm{L.o.t.}$ depends on $j$, but we do not make this explicit in the notation.} We moreover set $\hat F_{-1}(P,z_0)=\{0\}$ and $\hat F_{[-1]}(P,z_0)=\{0\}$. Thus, for $j\geq 0$, ${\rm L.o.t.}\colon\hat F_j(P,z_0)\to\hat F_{[j]}(P,z_0)$ is surjective with kernel $\hat F_{j-1}(P,z_0)$. We note that
\begin{equation}
\label{EqAsFjIncl}
  (z-z_0)\hat F_{j+1}(P,z_0)\subset\hat F_j(P,z_0)\subset\hat F_{j+1}(P,z_0),\qquad
  \hat F_{[j]}(P,z_0)\supset\hat F_{[j+1]}(P,z_0)
\end{equation}
for all $j$. For $\tilde u(z)=\sum_{k=0}^j(z-z_0)^{-k-1}u_k\in\hat F_j(P,z_0)$ we have
\begin{equation}
\label{EqAsFjRes}
  \Res_{z=z_0} \bigl(\rho^{i z}\tilde u(z)\bigr) = \sum_{k=0}^j \frac{i^k}{k!}\rho^{i z_0}(\log\rho)^k u_k \in \ker N(P),
\end{equation}
as follows by applying $N(P)$ under the integral sign to $\frac{1}{2\pi i}\oint_{z_0}\rho^{i z}\tilde u(z)\,\dd z$ where we integrate counterclockwise over a small circle around $z_0$. The space $\hat F_j(P,z_0)$ is thus isomorphic to the subspace
\begin{equation}
\label{EqAsFjPhys}
  F_j(P,z_0) := \left\{ u\in F(P,z_0),\ u=\sum_{k=0}^j\rho^{i z_0}(\log\rho)^k u_k \right\} = \ker_{F(P,z_0)} (\rho D_\rho-z_0)^{j+1}
\end{equation}
of $F(P,z_0)$ via $\hat F_j(P,z_0)\ni\tilde u\mapsto\Res_{z=z_0}(\rho^{i z}\tilde u(z))\in F_j(P,z_0)$. (In particular, if $\ord(P,z_0)=J<\infty$, then $\hat F_j(P,z_0)=\hat F_{J-1}(P,z_0)\cong F(P,z_0)$ for $j\geq J-1$.)

\begin{lemma}[Nondegenerate pairings]
\label{LemmaAsPair}
  Let $P\in\Diffb^m(M;E,F)$, and suppose that the principal symbol of $P$ is injective or surjective. Let $j\in\N_0$. Then the sesquilinear map
  \[
    b_j \colon \hat F_{[j]}(P,z_0) \times \hat F_{[j]}(P^*,\ol{z_0}+i w)\to\C,\qquad
    (u,u^*) \mapsto \la N(P,z)\tilde u(z),u^*\ra_{L^2(\pa M;F|_{\pa M})}\big|_{z=z_0},
  \]
  where $\tilde u\in \hat F_j(P,z_0)$ has ${\rm L.o.t.}(\tilde u)=u$, is well-defined (i.e.\ independent of the choice of $\tilde u$). Moreover:
  \begin{enumerate}
  \item\label{ItAsPairInj1} $b_j(u,u^*)=0$ for all $u^*\in \hat F_{[j]}(P^*,\ol{z_0}+i w)$ iff $u\in \hat F_{[j+1]}(P,z_0)$.
  \item\label{ItAsPairInj2} $b_j(u,u^*)=0$ for all $u\in \hat F_{[j]}(P,z_0)$ iff $u^*\in \hat F_{[j+1]}(P^*,\ol{z_0}+i w)$.
  \item\label{ItAsPairIso} The map $\hat F_{[j]}(P^*,\ol{z_0}+i w)/\hat F_{[j+1]}(P^*,\ol{z_0}+i w)\to(\hat F_{[j]}(P,z_0)/\hat F_{[j+1]}(P,z_0))^*$ induced by $u^*\mapsto b_j(-,u^*)$ is an (antilinear) isomorphism of finite-dimensional vector spaces.
  \end{enumerate}
\end{lemma}

This is closely related to \cite[Proposition~6.2]{MelroseAPS}. When $P$ and thus $N(P,z)$ is un\-der\-de\-ter\-mined-elliptic, then $\hat F_{[0]}(P,z_0)=\ker_{\CI(\pa M;E|_{\pa M})}N(P,z_0)$ is infinite-dimensional for all $z_0\in\C$ (and thus $\rank(P,z_0)=\infty$); see \cite{BourguignonEbinMarsdenPsdoSurj,BergerEbinTensors} and also the proof of Theorem~\ref{ThmAsIndetCon} below. Since for $J=\ord(P^*,\ol{z_0}+i w)$ we have $\hat F_{[j]}(P^*,\ol{z_0}+i w)=0$ for all $j\geq J$, part~\eqref{ItAsPairIso} of the Lemma implies that the space $\hat F_{[j]}(P,z_0)$ does not depend on $j$ for $j\geq J-1$ and has infinite dimension. Thus, $\ord(P,z_0)=\infty$ for all $z_0\in\C$.

\begin{proof}[Proof of Lemma~\usref{LemmaAsPair}]
  We consider the case that the principal symbol of $P$ is surjective; the injective case is analogous. Assuming that $J=\ord(P^*,\ol{z_0}+i w)\geq 1$, we then have a chain
  \[
    0 = \hat F_{[J]}(P^*,\ol{z_0}+i w) \subsetneq \hat F_{[J-1]}(P^*,\ol{z_0}+i w) \subseteq \cdots \subseteq \hat F_{[0]}(P^*,\ol{z_0}+i w) = \ker N(P^*,\ol{z_0}+i w)
  \]
  of finite-dimensional subspaces of $\CI(\pa M;F|_{\pa M})$.

  If $j=0$, then $\tilde u(z):=(z-z_0)^{-1}u$ is the unique choice of $\tilde u$, and one has $b_0(u,u^*)=\la\pa_z N(P,z_0)u,u^*\ra$. This vanishes for all $u^*\in \hat F_{[0]}(P^*,\ol{z_0}+i w)=\ker N(P^*,\ol{z_0}+i w)=\ker N(P,z_0)^*$ if and only if $\pa_z N(P,z_0)u\in\CI(\pa M;F|_{\pa M})$ lies in the range of $N(P,z_0)$ on $\CI(\pa M;E|_{\pa M})$; this uses that $N(P,z_0)$ has injective or surjective principal symbol. But the existence of $u_1\in\CI(M;E|_{\pa M})$ with $\pa_z N(P,z_0)u=-N(P,z_0)u_1$ is equivalent to $(z-z_0)^{-2}u+(z-z_0)^{-1}u_1\in\hat F_1(P,z_0)$, so $u\in\hat F_{[1]}(P,z_0)$. This proves~\eqref{ItAsPairInj1}. The argument for~\eqref{ItAsPairInj2} is analogous. Since~\eqref{ItAsPairInj1} implies that the map $\hat F_{[0]}(P,z_0)/\hat F_{[1]}(P,z_0)\to\hat F_{[0]}(P^*,\ol{z_0}+i w)^*$ induced by $b_1$ is injective, the domain is finite-dimensional (since the codomain is), and therefore its adjoint is surjective. This proves the surjectivity of the map in~\eqref{ItAsPairIso}, and its injectivity follows from~\eqref{ItAsPairInj2}.

  Suppose now we have established the Lemma for $j-1\geq 0$ in place of $j$. If $\tilde u,\tilde u'\in \hat F_j(P,z_0)$ have the same leading order term $u\in \hat F_{[j]}(P,z_0)$, then $\tilde u-\tilde u'\in \hat F_{j-1}(P,z_0)$; thus, for $u^*\in \hat F_{[j]}(P^*,\ol{z_0}+i w)$ we have $\la N(P,z)(\tilde u-\tilde u'),u^*\ra=b_{j-1}({\rm L.o.t.}(\tilde u-\tilde u'),u^*)=0$ by part~\eqref{ItAsPairInj2}. This shows that $b_j$ is well-defined.

  Given $u\in \hat F_{[j]}(P,z_0)$, choose now $\tilde u\in \hat F_j(P,z_0)$ with ${\rm L.o.t.}(\tilde u)=u=:u_j$. Suppose $b_j(u,u^*)=0$ for all $u^*\in \hat F_{[j]}(P^*,\ol{z_0}+i w)$. Then the map
  \begin{equation}
  \label{EqAsPairMap}
    \hat F_{[j-1]}(P^*,\ol{z_0}+i w) \ni u^* \mapsto \la N(P,z)\tilde u(z),u^*\ra|_{z=z_0} \in \C
  \end{equation}
  induces an element of $(\hat F_{[j-1]}(P^*,\ol{z_0}+i w)/\hat F_{[j]}(P^*,\ol{z_0}+i w))^*$; by part~\eqref{ItAsPairIso}, we may thus modify $\tilde u$ via addition of an suitable element of $\hat F_{j-1}(P,z_0)$ so as to ensure that~\eqref{EqAsPairMap} vanishes. Allowing next in~\eqref{EqAsPairMap} inputs $u^*\in \hat F_{[j-2]}(P^*,\ol{z_0}+i w)$, we obtain an element of $(\hat F_{[j-2]}(P^*,\ol{z_0}+i w)/\hat F_{[j-1]}(P^*,\ol{z_0}+i w))^*$, which upon adding a suitable element of $\hat F_{j-2}(P,z_0)$ to $\tilde u$ we can arrange to be $0$; and so on. Ultimately, we obtain a new $\tilde u(z)=\sum_{k=0}^j(z-z_0)^{-k-1}u_k\in \hat F_j(P,z_0)$, still with ${\rm L.o.t.}(\tilde u)=u$, so that $\la N(P,z)\tilde u(z),u^*\ra|_{z=z_0}=0$ for all $u^*\in \hat F_{[0]}(P^*,\ol{z_0}+i w)=\ker N(P^*,\ol{z_0}+i w)$. But this means that there exists $u_{-1}\in\CI(\pa M;E|_{\pa M})$ so that
  \[
    \sum_{k=0}^j \frac{1}{(k+1)!}\pa_z^{k+1}N(P,z_0)u_k = -N(P,z_0)u_{-1}.
  \]
  Therefore $\tilde v(z):=\sum_{k=0}^{j+1}(z-z_0)^{-k-1}u_{k-1}\in \hat F_{j+1}(P,z_0)$, and hence $u=u_j\in \hat F_{[j+1]}(P,z_0)$. Conversely, given $u\in \hat F_{[j+1]}(P,z_0)$, take $\tilde v(z)\in \hat F_{j+1}(P,z_0)$ of this form with ${\rm L.o.t.}(\tilde v)=u$; then $b_j(u,u^*)=\la N(P,z_0)(z-z_0)\tilde v(z),u^*\ra|_{z=z_0}=0$. This establishes part~\eqref{ItAsPairInj1}.

  Given $u^*\in \hat F_{[j]}(P^*,\ol{z_0}+i w)$, let $\tilde u^*(z)=\sum_{k=0}^j(z-(\ol{z_0}+i w))^{-k-1}u^*_k\in \hat F_j(P^*,\ol{z_0}+i w)$ with ${\rm L.o.t.}(\tilde u^*)=u^*_j=u^*$. Since $N(P^*,z)\tilde u^*(z)$ is holomorphic at $z=\ol{z_0}+i w$, the expression $N(P^*,\ol{z}+i w)\tilde u^*(\ol{z}+i w)$ is anti-holomorphic at $z=z_0$. Thus, for $\tilde u(z)\in \hat F_j(P,z_0)$ with ${\rm L.o.t.}(\tilde u)=u$, we obtain the symmetric expression
  \begin{align}
    b_j(u,u^*) &= \la N(P,z)\tilde u(z),(\ol{z}-\ol{z_0})^j\tilde u^*(\bar z+i w)\ra|_{z=z_0} \nonumber\\
  \label{EqAsPairAdj}
      &= \bigl((z-z_0)^j\la\tilde u(z),N(P,z)^*\tilde u^*(\ol{z}+i w)\ra\bigr)|_{z=z_0} \\
      &=\la u,N(P^*,\ol{z}+i w)\tilde u^*(\ol{z}+i w)\ra|_{z=z_0} \nonumber\\
      &= \la u,N(P^*,z)\tilde u^*(z)\ra|_{z=\ol{z_0}+i w}. \nonumber
  \end{align}
  Repeating the above arguments \emph{mutatis mutandis} shows~\eqref{ItAsPairInj2}. The statement~\eqref{ItAsPairIso} follows from~\eqref{ItAsPairInj1}--\eqref{ItAsPairInj2} and the finite-dimensionality of the $\hat F_{[j]}(P^*,\ol{z_0}+i w)$.
\end{proof}

\begin{cor}[Solvability of the normal operator]
\label{CorAsNorm}
  Suppose $P\in\Diffb^m(M;E,F)$ has surjective principal symbol. Let $z_0\in\C$ and put $J=\ord(P^*,\ol{z_0}+i w)$. Let $k\in\N_0$, and let $f_0,\ldots,f_k\in\CI(\pa M;F|_{\pa M})$. Then there exist $u_0,\ldots,u_{k+J}\in\CI(\pa M;E|_{\pa M})$, depending linearly and continuously on $(f_0,\ldots,f_k)$, so that $P u=f$ where
  \[
    u(\rho,y) = \sum_{j=0}^{J+k} \rho^{i z_0}(\log\rho)^j u_j(y),\qquad
    f(\rho,y) = \sum_{j=0}^k \rho^{i z_0}(\log\rho)^j f_j(y).
  \]
\end{cor}
\begin{proof}
  Relabeling and rescaling $f_j$ and $u_j$, we must show, in view of~\eqref{EqAsFjRes}, the existence of $u_j$ so that $N(P,z)u(z)-f(z)$ is holomorphic at $z=z_0$, where $f(z)=\sum_{j=0}^k(z-z_0)^{-j-1}f_j$ is given, and we seek $u(z)=\sum_{j=0}^{J+k}(z-z_0)^{-j-1}u_j$. By a simple induction on $k$, it suffices to prove this for $k=0$.

  We claim that there exists $\tilde u\in\hat F_{J-1}(P,z_0)$ (depending continuously on $f_0$) so that
  \begin{equation}
  \label{EqAsNorm}
    \la N(P,z)\tilde u(z),u^*\ra_{L^2(\pa M;F|_{\pa M})}|_{z=z_0}=\la f_0,u^*\ra_{L^2(\pa M;F|_{\pa M})}
  \end{equation}
  for all $u^*\in\ker (P,z_0)^*=\ker N(P^*,\ol{z_0}+i w)=\hat F_{[0]}(P^*,\ol{z_0}+i w)$; using such a $\tilde u$, we can then find $u_0\in\CI(\pa M;E|_{\pa M})$ (depending continuously on $f_0$, cf.\ Remark~\ref{RmkRightInv}) with $(f_0-N(P,z)\tilde u(z))|_{z=z_0}=N(P,z_0)u_0$ and conclude that $N(P,z)((z-z_0)^{-1}\tilde u(z)+(z-z_0)^{-1}u_0)-(z-z_0)^{-1}f_0$ is holomorphic at $z_0$.

  Requiring~\eqref{EqAsNorm} merely for all $u^*\in\hat F_{[J-1]}(P^*,\ol{z_0}+i w)$ is equivalent to $b_{J-1}(u_{J-1},-)=\la f_0,-\ra\in(\hat F_{[J-1]}(P^*,\ol{z_0}+i w))^*$ where $u_{J-1}={\rm L.o.t.}(\tilde u)$. By Lemma~\ref{LemmaAsPair}\eqref{ItAsPairIso}, this has a (unique) solution $u_{J-1}\in\hat F_{[J-1]}(P,z_0)$. Pick $\tilde u_{J-1}\in\hat F_{J-1}(P,z_0)$ with ${\rm L.o.t.}(\tilde u_{J-1})=u_{J-1}$, and let $f_1:=f_0-(N(P,z)\tilde u_{J-1}(z))|_{z=z_0}$. Consider then the equation
  \begin{equation}
  \label{EqAsNorm2}
    \la N(P,z)\tilde u_{J-2}(z),u^*\ra|_{z=z_0} = \la f_1,u^*\ra.
  \end{equation}
  For all $u^*\in\hat F_{[J-1]}(P^*,\ol{z_0}+i w)$, the left, resp.\ right hand side vanishes when $\tilde u_{J-2}\in\hat F_{J-2}(P,z_0)$ in view of Lemma~\ref{LemmaAsPair}\eqref{ItAsPairInj2} with $j=J-2$, resp.\ by construction of $f_1$. Solving~\eqref{EqAsNorm2} for $u^*\in\hat F_{[J-2]}(P^*,\ol{z_0}+i w)$ is thus equivalent to solving
  \[
    b_{J-2}(u_{J-2},-)=\la f_1,-\ra\in\bigl(\hat F_{[J-2]}(P^*,\ol{z_0}+i w)/\hat F_{[J-1]}(P^*,\ol{z_0}+i w)\bigr)^*
  \]
  for $u_{J-2}={\rm L.o.t.}(\tilde u_{J-2})\in\hat F_{[J-2]}(P,z_0)$ (where $\tilde u_{J-2}\in\hat F_{J-2}(P,z_0)$). Lemma~\ref{LemmaAsPair}\eqref{ItAsPairIso} provides us with a solution $u_{J-2}$ which is unique modulo $\hat F_{[J-1]}(P,z_0)$. We then set $f_2=f_1-(N(P,z)\tilde u_{J-2}(z))|_{z=z_0}$, etc. In this manner, we obtain $\tilde u_j\in F_j(P,z_0)$ for $j=J-1,\ldots,0$, and find for $\tilde u(z)=\sum_{j=0}^{J-1}\tilde u_j(z)$ that $f_0-(N(P,z)\tilde u(z))|_{z=z_0}$ is orthogonal to all $u^*\in\hat F_{[0]}(P^*,\ol{z_0}+i w)$, as required in~\eqref{EqAsNorm}. This completes the proof.
\end{proof}

\begin{prop}[Formal solution]
\label{PropAs}
  Let $P\in\Diffb^m(M;E,F)$ be an operator with surjective principal symbol. Let $\cF\subset\C\times\N_0$ be an index set. Define the index set $\cE(P,\cF)\subset\C\times\N_0$ by
  \begin{equation}
  \label{EqAsInd}
    \cE(P,\cF) := \left\{ (z+j,k+\ell) \colon (z,k)\in\cF,\ j\in\N_0,\ \ell\leq\sum_{q=0}^j \ord\left(P^*,\ol{-i(z+q)}+i w\right)\right\}.
  \end{equation}
  Then for all $f\in\cA_\phg^\cF(M;F)$, there exists $u\in\cA_\phg^{\cE(P,\cF)}(M;E)$, depending continuously on $f$, so that $P u-f\in\CIdot(M;F)$.
\end{prop}

This generalizes (and in the elliptic setting also sharpens) \cite[Lemma~5.44]{MelroseAPS}. In the applications discussed in~\S\ref{SA}, the orders of $P^*$ are zero at all but finitely many points in $\C$, and therefore $\cE$ is only a modest enlargement of $\cF$. We conjecture that for generic $P,f$, the index set $\cE$ is the smallest one for which the conclusion holds.

\begin{proof}[Proof of Proposition~\usref{PropAs}]
  Since the result is local near $\pa M$, we may work in a collar neighborhood of $\pa M$ and assume that $f=\chi f$ where $\chi\in\CIc([0,1)_\rho\times\pa M)$. Moreover, by a Borel lemma argument, it suffices to consider the case that $f$ is replaced by any individual term of its polyhomogeneous expansion, so $f=\chi f_0$ where $f_0(\rho,y)=\rho^z(\log\rho)^k f_{(z,k)}(y)$ with $f_{(z,k)}\in\CI(\pa M;F|_{\pa M})$. By Corollary~\ref{CorAsNorm}, there exist $u_{0,0},\ldots,u_{0,k+j_0}\in\CI(\pa M;E|_{\pa M})$, $j_0=\ord(P^*,\ol{-i z}+i w)$, so that $N(P)u_0=f_0$ for $u_0=\sum_{j=0}^{k+j_0} \rho^z(\log\rho)^j u_{0,j}$. Therefore,
  \[
    f-P(\chi u_0)=\chi(f_0-P u_0) - [P,\chi]u_0 = -\chi(P-N(P))u_0 - [P,\chi]u_0
  \]
  is the sum of $\chi f_1$, $f_1:=-(P-N(P))u_0\in\cA_\phg^{\cF_1}([0,1)\times\pa M)$, and the commutator term which has compact support in $M^\circ$; here $\cF_1=\{(z+1+j,k')\colon j\in\N_0,\ k'\leq k+j_0\}$.

  We can then similarly solve away $f_1$ term by term to leading order, producing $u_1=\sum_{j=0}^{k+j_0+j_1}\rho^{z+1}(\log\rho)^j u_{1,j}$ with $u_{1,j}\in\CI(\pa M;E|_{\pa M})$, $j_1=\ord(P^*,\ol{-i(z+1)}+i w)$, so that $N(P)u_1=f_1$; and so on. Taking $u$ to be an asymptotic sum of $\chi u_0,\chi u_1,\ldots$ finishes the proof.
\end{proof}

\subsection{Solutions for rapidly decaying forcing; proof of the main result}
\label{SsAsSchwartz}

It remains to examine the solvability of $P u=f$ for rapidly vanishing $f$.

\begin{prop}[Solution for Schwartz forcing]
\label{PropAsSchw}
  Suppose $P\in\Diffb^m(M;E,F)$ has surjective principal symbol. Let $f\in\CIdot(M;F)$. Suppose there exists $\alpha\in\R$ so that
  \[
    \la f,u^*\ra_{L^2(M,\mu;F)}=0\qquad \text{for all}\quad u^*\in\ker P^*\cap\cA^{-\alpha-w}(M;F),
  \]
  and let $\alpha_0\in\R\cup\{+\infty\}$ be the supremum of all such $\alpha$. Set
  \begin{equation}
  \label{EqAsSchwInd}
  \begin{split}
    \cE(P,\alpha_0) &:= \Biggr\{ (z+j,k+\ell) \colon k\leq\bar k,\ (z,\bar k)\in\surjSpecb(P),\ \Re z\geq\alpha_0, \\
      &\hspace{10em} \ell\leq\sum_{q=1}^j \ord\left(P^*,\ol{-i(z+q)}+i w\right) \Biggr\}.
  \end{split}
  \end{equation}
  Then there exists $u\in\cA_\phg^{\cE(P,\alpha_0)}(M;E)$ (when $\alpha_0=+\infty$, this means $u\in\CIdot(M;E)$) so that $P u=f$.
\end{prop}

By Lemma~\ref{LemmaAsKer} applied to $P^*$, the existence of $\alpha$ is guaranteed for all $f\in\CIdot(M;F)$ if and only if no $u^*\in\ker P^*$ vanishes to all orders at $\pa M$ (i.e.\ $u^*\in\CIdot(M;F)$, $P^*u^*=0$ implies $u^*=0$), i.e.\ if and only if unique continuation at $\pa M$ holds for elements of the nullspace of $P^*$.

\begin{rmk}[Smooth solvability for finite-dimensional families]
\label{RmkAsSchwParam}
  If $f$ depends smoothly on a finite-dimensional parameter (similarly to Proposition~\ref{PropACI}), and satisfies the assumptions of Proposition~\usref{PropAsSchw} for all values of the parameter (with $\alpha_0$ fixed), then one can find a solution $u$ of the stated class which moreover depends smoothly on this parameter.
\end{rmk}

We prepare the proof of Proposition~\ref{PropAsSchw} with the following technical result:

\begin{lemma}[Nondegenerate pairings \#2]
\label{LemmaAsPair2}
  Suppose $P\in\Diffb^m(M;E,F)$ has surjective principal symbol. Let $z_0\in\C$ and $J=\ord(P^*,\ol{z_0}+i w)$. Define the sesquilinear map
  \begin{alignat*}{3}
    b \colon\hat F_{J-1}(P,z_0) &\,\times\,&& \hat F_{J-1}(P^*,\ol{z_0}+i w) &&\to \C, \\
    (\tilde u&,&&\tilde u^*) &&\mapsto \Res_{z=z_0}\la N(P,z)\tilde u(z),\tilde u^*(\bar z+i w)\ra_{L^2(\pa M;F|_{\pa M})}.
  \end{alignat*}
  Then the linear map $\hat F_{J-1}(P,z_0)\ni\tilde u\mapsto b(\tilde u,-)$, taking values in the space of antilinear functionals on $\hat F_{J-1}(P^*,\ol{z_0}+i w)$, is surjective and has a continuous linear right inverse.
\end{lemma}
\begin{proof}
  Since $\hat F_{J-1}(P^*,\ol{z_0}+i w)$ is finite-dimensional, it suffices to prove that the adjoint map $\tilde u^*\mapsto b(-,\tilde u^*)$ is injective. Let thus $\tilde u^*\in\hat F_{J-1}(P^*,\ol{z_0}+i w)$ be such that
  \[
    b(\tilde u,\tilde u^*)=0\quad\forall\ \tilde u\in\hat F_{J-1}(P,z_0).
  \]
  For $\tilde u=(z-z_0)^{-1}u$, $u\in\hat F_{[J-1]}(P,z_0)$, we have $0=b(\tilde u,\tilde u^*)=b_{J-1}(u,{\rm L.o.t.}(\tilde u^*))$ in the notation of Lemma~\ref{LemmaAsPair} (see also~\eqref{EqAsPairAdj}); by Lemma~\ref{LemmaAsPair}, this implies ${\rm L.o.t.}(\tilde u^*)\in\hat F_{[J]}(P^*,\ol{z_0}+i w)=\{0\}$. Therefore, $\tilde u^*\in\hat F_{J-2}(P^*,\ol{z_0}+i w)$. Repeating this argument with $\tilde u(z)=(z-z_0)^{-1}u$, $u\in\hat F_{[J-2]}(P,z_0)$, implies ${\rm L.o.t.}(\tilde u^*)\in\hat F_{[J-1]}(P^*,\ol{z_0}+i w)$, i.e.\ we can write
  \begin{equation}
  \label{EqAsPair2Dec}
    \tilde u^*(z) = (z-(\ol{z_0}+i w))\tilde u^*_{J-1}(z) + \tilde u^*_{J-3}(z),\qquad \tilde u^*_\ell\in\hat F_\ell(P^*,\ol{z_0}+i w)\ \ (\ell=J-1,J-3).
  \end{equation}
  For $\tilde u(z)=(z-z_0)^{-1}u$, $u\in\hat F_{[J-3]}(P,z_0)$, we find
  \[
    0=b(\tilde u,\tilde u^*)=b_{J-3}(u,{\rm L.o.t.}(\tilde u^*_{J-3}))+b_{J-2}\bigl(u,{\rm L.o.t.}((z-(\ol{z_0}+i w))\tilde u_{J-1}^*)\bigr);
  \]
  by Lemma~\ref{LemmaAsPair}, the second term vanishes, and thus the vanishing of the first term implies ${\rm L.o.t.}(\tilde u^*_{J-3})\in\hat F_{[J-2]}(P^*,\ol{z_0}+i w)$. Therefore, $\tilde u^*_{J-3}\in(z-(\ol{z_0}+i w))\hat F_{J-2}(P^*,\ol{z_0}+i w)+\hat F_{J-4}(P^*,\ol{z_0}+i w)$, which implies that we can improve the decomposition~\eqref{EqAsPair2Dec} to $\tilde u^*(z)=(z-(\ol{z_0}+i w))\tilde u^*_{J-1}(z)+\tilde u^*_{J-4}(z)$ for suitable $\tilde u^*_\ell\in\hat F_\ell(P^*,\ol{z_0}+i w)$ ($\ell=J-1,J-4$). Continuing in this fashion, we ultimately find that we can write $\tilde u^*(z)=(z-(\ol{z_0}+i w))\tilde u^*_{J-1}(z)$, $\tilde u^*_{J-1}\in\hat F_{J-1}(P^*,\ol{z_0}+i w)$. We have thus shown that $\tilde u^*$ is one order less singular at $z=\ol{z_0}+i w$ than assumed initially.

  Repeating the above arguments for $\tilde u=(z-z_0)^{-2}u$ where $u\in\hat F_j(P,z_0)$, $j=J-1,J-2,\ldots$ implies $\tilde u^*\in(z-(\ol{z_0}+i w))^2\hat F_{J-1}(P^*,\ol{z_0}+i w)$. Continuing in this fashion until the exponent $2$ is improved to $J$, we get $\tilde u^*=0$, finishing the proof.
\end{proof}

\begin{proof}[Proof of Proposition~\usref{PropAsSchw}]
  If $\alpha_0=+\infty$, the claim follows from Corollary~\ref{CorbSchwartz}, so let us assume that $\alpha_0<\infty$. Let $\alpha_0<\alpha_1<\ldots$ be the finite or countably infinite sequence of real numbers with
  \[
    \{\Re z \colon z\in\surjspecb(P),\ \Re z\geq\alpha_0\}=\{\alpha_0,\alpha_1,\ldots\}.
  \]
  For $j\geq 0$, let $\cK^*_j:=\{u^*\in\bigcap_{\eps>0}\cA^{-\alpha_j-w-\eps}(M;F)\colon P^*u^*=0\}$. (Each $\cK_j^*$ is finite-dimensional by Proposition~\ref{PropbFred}, but the union of all $\cK_j^*$ may be infinite-dimensional.) We need to show that there exists $u'$ with $u'\in\cA_\phg^{\cE(P,\alpha_0)}(M;E)$, $P u'\in\CIdot(M;F)$, and so that
  \begin{equation}
  \label{EqAsSchw}
    \la P u',u^*\ra_{L^2(M,\mu'F)}=\la f,u^*\ra_{L^2(M,\mu;F)}
  \end{equation}
  for all $u^*\in\cK_j^*$, $j\geq 0$. Given such a $u'$, we then conclude that $f-P u'\in\CIdot(M;F)$ is orthogonal to $\ker_{\CmI(M;F)}P^*=\bigcup_{j\geq 0}\cK_j^*$. By Corollary~\ref{CorbSchwartz}, we then have $f-P u'=P u''$ for some $u''\in\CIdot(M;E)$, and hence $u=u'+u''$ is of the required form.

  We now turn to the task of finding $u'$ so that~\eqref{EqAsSchw} holds. Let $\chi\in\CIc([0,1)_\rho\times\pa M)$ be identically $1$ near $\rho=0$. Define the finite set $\{(z_q,k_q)\}=\{(z,k)\colon z\in\surjspecb(P)\colon\Re z=\alpha_0,\ k=\ord(P^*,\bar z+i w)-1\}$, and let
  \[
    G(P^*,-\alpha_0-w) := \bigoplus_q F(P^*,\ol{-i z_q}+i w).
  \]
  (This is a finite-dimensional vector space.) Consider first $u^*\in\cK_0^*$. By Lemma~\ref{LemmaAsKer}, we have $u^*=\sum_q \chi u^*_q+u^{*\prime}$ where $u_q^*\in F(P^*,\ol{-i z_q}+i w)$ and $u^{*\prime}\in\cA^{-\alpha_0-w+\delta}$ for some $\delta>0$. We have $u_q^*=0$ for all $q$ if and only if $u^*=u^{*\prime}\in\cK_{-1}^*:=\bigcup_{\eps>0}\ker_{\cA^{-\alpha_0-w+\eps}(M;F)}P^*$; and in this case, we have $\la f,u^*\ra=0$ by assumption. Therefore, we have an injective map $\cK_0^*/\cK_{-1}^*\to G(P^*,-\alpha_0-w)$, $[u^*]\mapsto(u^*_q)$, and the antilinear functional $\la f,-\ra_{L^2(M,\mu;F)}$ on $\cK_0^*$ induces an antilinear functional on $\cK_0^*/\cK_{-1}^*$ which we can then extend to an antilinear functional $\lambda_f$ on $G(P^*,-\alpha_0-w)$.

  Define now
  \[
    G(P,\alpha_0) := \bigoplus_q F_{k_q}(P,-i z_q),
  \]
  and consider for $v=(v_q)\in G(P,\alpha_0)$ (identified with the finite sum $\sum_q v_q$, which is polyhomogeneous on $[0,1)\times\pa M$) and $u^*\in\cK_0^*$ the pairing $\la P(\chi v),u^*\ra_{L^2(M,\mu;F)}$. Defining the approximate identity $\phi_\eps=1-\chi_\eps$ where $\chi_\eps(\rho,y)=\chi(\frac{\rho}{\eps},y)$, we rewrite this (using that $\supp\dd\phi_\eps\subset\chi^{-1}(1)$ for sufficiently small $\eps>0$) as
  \begin{align*}
    \la P(\chi v),u^*\ra &= \lim_{\eps\searrow 0}\left( \la P(\chi v),\phi_\eps u^*\ra - \la \chi v,\phi_\eps P^*u^*\ra\right) = -\lim_{\eps\searrow 0} \la[P,\phi_\eps]\chi v,u^*\ra \\
      &=-\lim_{\eps\searrow 0} \la[N(P),\phi_\eps]\chi v,\chi u^*\ra = \lim_{\eps\searrow 0} \left( \la \phi_\eps N(P)(\chi v),\chi u^*\ra - \la\phi_\eps \chi u,N(P^*)(\chi u^*)\ra \right) \\
      &= \la N(P)(\chi v),\chi u^*\ra - \la \chi v, N(P^*)(\chi u^*)\ra.
  \end{align*}
  This vanishes for $u^*\in\cK_{-1}^*$, as follows by integrating by parts on the left hand side. Moreover, the final expression in the first line implies that this pairing only depends on the leading order terms of $v$ and $u^*$. Therefore, for $v\in G(P,\alpha_0)$ we obtain another linear functional on $G(P^*,-\alpha_0-w)$. The heart of the proof is thus to show that the map
  \begin{equation}
  \label{EqAsSchwPair}
    G(P,\alpha_0) \ni v \mapsto B(v,\cdot) := \la N(P)(\chi v),\chi\cdot\ra_{L^2(M,\mu;F)} - \la\chi v,N(P^*)(\chi\cdot)\ra_{L^2(M,\mu;F)}
  \end{equation}
  into the space of antilinear functionals on $G(P^*,-\alpha_0-w)$ is surjective. Once this is shown, we select $v$ so that $B(\cdot,v)=\lambda_f$; this implies for $u'=\chi v$ the validity of~\eqref{EqAsSchw} for all $u^*\in\cK_0^*$. In order to remedy the defect that $P(\chi v)$ typically does not vanish to infinite order at $\pa M$, we apply Proposition~\ref{PropAs} to the forcing term $-P(\chi v)=-\chi N(P)v-\chi(P-N(P))v=-\chi(P-N(P))v$ to find $v'\in\cA_\phg^{\cE(P,\alpha_0)}(M;E)\cap\bigcap_{\eps>0}\cA^{\alpha_0+1-\eps}(M;E)$ so that $P u'_0\in\CIdot(M;F)$ where $u'_0=\chi v+v'\in\cA_\phg^{\cE(P,\alpha_0)}(M;E)$. We still have $\la P u'_0,u^*\ra_{L^2(M,\mu;F)}=\la f,u^*\ra_{L^2(M,\mu;F)}$ for all $u^*\in\cK_0^*$ since the contribution of $v'$ to the pairing vanishes (via integration by parts). But this means that $f-P u'_0\in\CIdot(M;F)$ satisfies the same hypotheses as $f$, except now $\la f-P u'_0,u^*\ra_{L^2(M,\mu;F)}=0$ holds not only for all $u^*\in\cK_{-1}^*$, but for all $u^*\in\cK_0^*$. An inductive procedure gives $u'_j\in\cA_\phg^{\cE(P,\alpha_j)}(M;E)\subset\cA_\phg^{\cE(P,\alpha_0)}(M;E)$ so that $\la f-P(u'_0+\cdots+u'_j),u^*\ra=0$ for all $u^*\in\cK_j^*$. Taking $u'\in\cA_\phg^{\cE(P,\alpha_0)}(M;E)$ to be an asymptotic sum of the $u'_j$, $j\in\N_0$, arranges~\eqref{EqAsSchw} for all $u^*\in\bigcup_{j\in\N_0}\cK_j^*$.

  It remains to show the surjectivity of~\eqref{EqAsSchwPair}. Note that $N(P)(\chi v)$, resp.\ $N(P^*)(\chi u^*)$ vanishes to one order more at $\pa M$ than $\chi v\in\cA^{\alpha_0-\eps}(M;E)$ and $\chi u^*\in\cA^{-\alpha_0-w-\eps}(M;E)$. Using the Mellin transform, Plancherel's theorem and the fact that $\wh{\chi v}(z)$ and $\wh{\chi u^*}(z)$ are meromorphic and vanish rapidly at real infinity allow us to write
  \begin{align*}
    B(v,u^*)&=\frac{1}{2\pi} \oint_\gamma \big\la \bigl(N(P)(\chi v)\bigr)\ftrans(z), \wh{\chi u^*}(\bar z+i w)\big\ra_{L^2(\pa M,\nu;F|_{\pa M})}\,\dd z \\
    &\qquad = \frac{1}{2\pi} \oint_\gamma \big\la \wh{\chi v}(z), \bigl(N(P^*)(\chi u^*)\bigr)\ftrans(\bar z+i w)\big\ra_{L^2(\pa M,\nu;F|_{\pa M})}\,\dd z,
  \end{align*}
  where $\nu>0$ is the unique density on $\pa M$ with the property that $\rho^{-w}\mu(\rho\pa_\rho,\cdot)=\nu$ at $\rho=0$, and where $\gamma$ is the union of $\{\Im z=-\alpha_0-\eps\}$ (traversed in the direction of increasing real part) and $\{\Im z=-\alpha_0+\eps\}$ (traversed in the direction of decreasing real part). If $v\in F_{k_q}(P,-i z_q)$ and $u^*\in F(P^*,\ol{-i z_{q'}}+i w)$ with $q\neq q'$, the integrand is holomorphic and thus $B(v,u^*)=0$. Thus, $B$ is block-diagonal. For $q=q'$ on the other hand, we have
  \[
    B(v,u^*) = \frac{1}{2\pi}\oint_{z_q} \big\la \bigl(N(P)(\chi v)\bigr)\ftrans(z), \wh{\chi u^*}(\bar z+i w)\big\ra_{L^2(\pa M,\nu;F|_{\pa M})}\,\dd z = i b(\tilde v,\tilde u^*)
  \]
  in the notation of Lemma~\ref{LemmaAsPair2}, where $\tilde v\in\hat F_{k_q}(P,-i z_q)$ and $\tilde u^*\in\hat F_{k_q}(P^*,\ol{-i z_q}+i w)$ correspond to $v$ and $u^*$ via~\eqref{EqAsFjRes}. (That is, $\wh{\chi v}(z)-\tilde v(z)$ is holomorphic at $z=z_q$, and similarly $\wh{\chi u^*}(z)-\tilde u^*(z)$ is holomorphic at $z=\ol{z_q}+i w$.) Here, the integration contour is a small circle around $z_q$, traversed counter-clockwise. By Lemma~\ref{LemmaAsPair2}, $\tilde v\mapsto i b(\tilde v,-)$ surjects onto the space of antilinear functionals on $F_{k_q}(P^*,\ol{-i z_q}+i w)^*$. This completes the proof.
\end{proof}

Combining Propositions~\ref{PropAs} and \ref{PropAsSchw}, we now obtain:

\begin{thm}[Solution with sharp asymptotics]
\label{ThmAs}
  Suppose $P\in\Diffb^m(M;E,F)$ has surjective principal symbol. Let $\cF\subset\C\times\N_0$ be an index set, and define the index sets $\cE(P,\cF)$ and $\cE(P,\alpha_0)$ (for $\alpha_0\in\R\cup\{+\infty\}$) by~\eqref{EqAsInd} and \eqref{EqAsSchwInd}; put $\alpha_\cF:=\min_{(z,0)\in\cF}\Re z$. Let $f\in\cA_\phg^\cF(M;F)$. Let $-\infty\leq\alpha_{\rm coker}\leq\alpha_\cF$ be the supremum of all $\alpha<\alpha_\cF$ so that $\la f,u^*\ra_{L^2(M,\mu;F)}=0$ for all $u^*\in\ker_{\cA^{-\alpha-w}(M;F)}P^*$. Suppose that $\alpha_{\rm coker}>-\infty$. Finally, if $\alpha_{\rm coker}\in\Re\surjspecb(P)$, set $\alpha_0:=\alpha_{\rm coker}$; otherwise, let $\alpha_0\in[\alpha_{\coker},+\infty]$ be the upper bound of the largest interval with lower bound $\alpha_{\coker}$ which is disjoint from $\Re\surjspecb(P)$. Then there exists $u\in\cA_\phg^{\cE(P,\cF)\cup\cE(P,\alpha_0)}(M;E)$ with $P u=f$.
\end{thm}

The only assumption on $f$ in this result is the finiteness of $\alpha_{\rm coker}$. This is always satisfied if and only if unique continuation at $\pa M$ holds for elements of the smooth nullspace of $P^*$, as already discussed after Proposition~\ref{PropAsSchw}.

In the case that one can take $\alpha_0=+\infty$, we have $\cE(P,\alpha_0)=\emptyset$, so $u\in\cA_\phg^{\cE(P,\cF)}(M;E)$. In the further special case that $\cF=\emptyset$, one obtains a solution $u\in\CIdot(M;E)$. Thus, Theorem~\ref{ThmAs} generalizes (but of course its proof relies on) Corollary~\ref{CorbSchwartz}.

We also note that if $f$ depends smoothly on a finite-dimensional parameter, then one can find $u$ of the stated class with smooth dependence on the parameter, similarly to Proposition~\ref{PropACI} and Remark~\ref{RmkAsSchwParam}.

\begin{proof}[Proof of Theorem~\usref{ThmAs}]
  First, we use Proposition~\ref{PropAs} to find $u_0\in\cA_\phg^{\cE(P,\cF)}(M;E)$ so that $f_1:=f-P u_0\in\CIdot(M;F)$. Since $\min_{(z,0)\in\cE(P,\cF)}\Re z=\min_{(z,0)\in\cF}\Re z=\alpha_\cF\geq\alpha_{\coker}$, an integration by parts shows that $\la f_1,u^*\ra=0$ for all $u^*\in\ker_{\cA^{-\alpha-w}(M;F)}P^*$, $\alpha<\alpha_{\coker}$. In the case that $\alpha_0>\alpha_{\rm coker}$, the kernel $\ker_{\cA^{-\alpha-w}(M;F)}P^*$ is independent of $\alpha\in[\alpha_{\rm coker},\alpha_0)$. Therefore, Proposition~\ref{PropAsSchw} applies and gives $u_1\in\cA_\phg^{\cE(P,\alpha_0)}(M;E)$ with $f_1=P u_1$. Setting $u=u_0+u_1$ finishes the proof.
\end{proof}

\begin{rmk}[Comparison with $P P^*$ arguments]
\label{RmkAsPPstar}
  Fix on $M$ a positive density with weight $0$. Using the notation of Theorem~\ref{ThmAs}, fix any $\alpha<\alpha_{\rm coker}$ with $\alpha\neq\Re\surjspecb(P)$; then $f\in\Hb^{\infty,\alpha}(M;F)$ is orthogonal to the nullspace of $P^*$ on $\Hb^{\infty,-\alpha}(M;F)$. We may then attempt to solve $P u=f$ via~\eqref{EqbRIT} by means of inverting the elliptic operator $T_\alpha$. The assumptions on $f,\alpha$ imply that this elliptic equation indeed has a solution $v\in\Hb^\infty(M;F)$, and thus $u\in\Hb^{\infty,\alpha}(M;E)$. (In other words, $u=G f$ in the notation of Corollary~\ref{CorbRI}.) Furthermore, $v$ is polyhomogeneous since $f$ is, and the index set of $v$ is enlarged relative to that of $f$ by adding elements related to $\Specb(T_\alpha)$, cf.\ \cite[Proposition~5.61]{MelroseAPS}. The index set of $v$ and thus of $u$ depends in a complicated manner on the choice of $\alpha$; and indeed the polyhomogeneous expansion of $u$ at $\pa M$ typically has many more terms than the solution produced by Theorem~\ref{ThmAs}. An example of this phenomenon is given in the introduction around equation~\eqref{EqIdiv}. See also \cite[Proposition~4.10, 4.14, and Remark~4.15]{HintzGlueID}.
\end{rmk}

\begin{rmk}[Parametrices and ps.d.o.s]
\label{RmkAsPx}
  We do not address here the interesting problem to construct, in the context of Theorem~\ref{ThmAs}, a right parametrix, or indeed a (generalized) right inverse, of $P$ in the large b-calculus \cite{MelroseAPS} whose index sets are as small as possible. Theorem~\ref{ThmAs}, applied to effect the right inversion of the normal operator at the front face of the b-double space of $M$, is likely useful for this purpose. The main benefit of a parametrix would be that it gives mapping properties on coarser function spaces (such as weighted b-Sobolev or H\"older spaces); but for such purposes, Corollary~\ref{CorbRI} is typically sufficient in applications. We also do not consider here the problem of generalizing our arguments to the case that $P$ is a right elliptic b-\emph{pseudodifferential} operator; this largely only requires notational changes.
\end{rmk}

\subsection{Infinite-dimensionality of the kernel}
\label{SsAKer}

Complementing Theorem~\ref{ThmAs}, we now prove two results which show that the \emph{kernel} of $P$ on conormal or polyhomogeneous functions is infinite-dimensional (and thus solutions of $P u=f$ are non-unique) in rather dramatic ways.

\begin{thm}[Infinite-dimensional conormal nullspace]
\label{ThmAsIndetCon}
  Let $P\in\Diffb^m(M;E,F)$ be underdetermined-elliptic. For $\alpha\in\R$, set $K^\alpha=\ker_{\cA^\alpha(M;E)}P$. Then for all $\alpha<\beta$, the space $K^\alpha/K^\beta$ is infinite-dimensional.
\end{thm}

Theorem~\ref{ThmAsIndetCon} is a special case of Theorem~\ref{ThmAsIndetPhg} below; we include it nonetheless, since it (and its proof) generalizes to a larger class of operators $P$, such as uniform differential operators on manifolds with cylindrical ends $[1,\infty)_r\times Y$ which asymptote to an $r$-translation invariant operator at an exponential rate as $r\to\infty$ and have (uniformly) surjective principal symbols. (The relationship is via $\rho=e^{-r}$, so $\pa_r=-\rho\pa_\rho$, with totally characteristic operators having smooth coefficients in $\rho\geq 0$, whereas uniform operators have coefficients which are uniformly bounded on $[1,\infty)\times Y$ together with all derivatives along $\pa_r$.)

\begin{proof}[Proof of Theorem~\usref{ThmAsIndetCon}]
  Fix a density of weight $0$ on $M$ which in a collar neighborhood of $\pa M$ is given by $|\frac{\dd\rho}{\rho}\nu|$ where $\nu$ is a positive density on $\pa M$. Since $\Hb^{\infty,\alpha}(M;E)\subset\cA^\alpha(M;E)\subset\Hb^{\infty,\alpha-\eps}(M;E)$ for all $\eps>0$, it suffices to prove the claim for $K^{s,\alpha}/K^{s,\beta}$, $K^{s,\alpha}:=\ker_{\Hb^{s,\alpha}(M;E)}P$, when $s=\infty$. By subdividing the interval $(\alpha,\beta)$ into any finite number of nonempty subintervals, it further suffices to prove that $K^{\infty,\alpha}/K^{\infty,\beta}$ is nontrivial, i.e.\ that $K^{\infty,\alpha}\supsetneq K^{\infty,\beta}$. Furthermore, we may increase $\alpha$ and decrease $\beta$ by arbitrarily small amounts to as to ensure that $\alpha,\beta\notin\Re\surjspecb(P)$. Finally, upon conjugating $P$ by $\rho^{-\beta}$, we may reduce to the case that $\alpha<0=\beta$.

  We have $\Hb^m(M;E)=\ker_{\Hb^m(M;E)}P\oplus P^*(\Hb^{2 m}(M;F))$ since both summands are closed by Proposition~\ref{PropbFred}. Moreover, there exists $C>0$ so that
  \begin{equation}
  \label{EqAsIndetCon}
    \|u\|_{\Hb^m(M;E)} \leq C\|P u\|_{\Hb^0(M;E)},\qquad u\in P^*(\Hb^{2 m}(M;F)).
  \end{equation}
  We similarly have the splitting $\Hb^\infty(M;E)=K^{\infty,0}\oplus P^*(\Hb^\infty(M;F))$, for if $\Hb^\infty(M;E)\ni u=u'+P^*u''$ with $u'\in\Hb^m(M;E)$, $P u'=0$, and $u''\in\Hb^{2 m}(M;F)$, then $P u=P P^*u''\in\Hb^\infty(M;F)$ implies $u''\in\Hb^\infty(M;F)$, and thus also $u'=u-P^*u''\in\Hb^\infty(M;E)$. Moreover, $P^*(\Hb^\infty(M;F))$ is closed since  $u_j\in\Hb^\infty(M;F)$, $P^*u_j\to f\in\Hb^\infty(M;E)$ implies the existence of $u\in\Hb^{2 m}(M;F)$ with $f=P^*u$, and thus $u\in\Hb^\infty(M;F)$ by elliptic regularity. 

  Suppose now, for the sake of contradiction, that $K^{\infty,\alpha}=K^{\infty,0}$. Then the continuous inclusion map $(K^{\infty,0},\|\cdot\|_{\Hb^{\infty,0}(M;E)})\to(K^{\infty,\alpha},\|\cdot\|_{\Hb^{\infty,\alpha}(M;E)})$ is a bijection of Fr\'echet spaces, and hence an isomorphism; therefore, there exists $m'\in\R$ (which may well be larger than $m$) so that
  \[
    \|u\|_{\Hb^m(M;E)} \leq C\|u\|_{\Hb^{m',\alpha}(M;E)},\qquad u\in K^{\infty,0}.
  \]
  Together with~\eqref{EqAsIndetCon} for $u\in P^*(\Hb^\infty(M;F))$, we obtain
  \begin{equation}
  \label{EqAsIndetEst}
    \|u\|_{\Hb^m(M;E)} \leq C\left(\|P u\|_{\Hb^0(M;E)} + \|u\|_{\Hb^{m',\alpha}(M;E)}\right),\qquad u\in\Hb^m(M;E).
  \end{equation}
  The proof is complete once we show that this implies the injectivity of $\sigmab(P)$ over $\pa M$. To this end, we work in a local chart $\R^{n-1}_y$ on $\pa M$. Let $\phi\in\CIc((0,1))$ and $\psi\in\CIc(\pa M;E|_{\pa M})$ with support in the chart so that $\int_0^1|\phi(\rho)|^2\,\frac{\dd\rho}{\rho}=1$ and $\|\psi\|_{L^2(\pa M;E|_{\pa M})}=1$. Let $(0,0)\neq(\xi,\eta)\in\R\times\R^{n-1}$. We apply the estimate~\eqref{EqAsIndetEst} to $u_{\delta,\lambda}(\rho,y)=\phi(\frac{\rho}{\delta})\psi(y)\rho^{i\lambda\xi}e^{i y\cdot\lambda\eta}$ where $\delta\in(0,1)$ and $\lambda>1$; there then exist $c,C'>0$ (with $c$ independent of $\phi,\psi$) so that
  \begin{align*}
    c\lambda^m - C'\lambda^{m-1} &\leq C\biggr( \left\| \sigmab^m(P)(\rho,y;\lambda\xi,\lambda\eta)\phi\left(\frac{\rho}{\delta}\right)\psi(y)\right\|_{L^2(M;F)} + C'\lambda^{m-1} + C'\delta^{-\alpha}\lambda^{m'} \biggr).
  \end{align*}
  Here we use that $\rho^{-\alpha}\phi(\frac{\rho}{\delta})=\delta^{-\alpha}(\frac{\rho}{\delta})^{-\alpha}\phi(\frac{\rho}{\delta})$ has norm bounded by $C'\delta^{-\alpha}$ as a multiplication operator on any weighted b-Sobolev space. Take $\delta=\lambda^{-\gamma}$ where $\gamma>0$ is chosen large enough such that $m'+\alpha\gamma<m$.\footnote{The estimate~\eqref{EqAsIndetEst} can be viewed as an estimate for $u$ on second microlocal b-Sobolev spaces $H_{\bop,\gamma}^{m,\nu,\alpha}(M;E)$ which are defined via testing with b-pseudodifferential operators whose symbols are conormal on an non-homogeneous (the degree being $\gamma>0$) blow-up of $\ol{\Tb^*}M$ at $\Sb^*_{\pa M}M$, with weight $m,\nu,\alpha$ at the lift of $\Sb^*M$, the front face, and the lift of $\ol{\Tb^*_{\pa M}}M$, respectively. (Such spaces with $\gamma=1$, were introduced in \cite{VasyLowEnergy}). Then the first two orders of $\Hb^{m',\alpha}(M;E)=H_{\bop,2}^{m',m'+\gamma\alpha,\alpha}(M;E)$ are less than the corresponding orders of $\Hb^m(M;E)=H_{\bop,2}^{m,m,\alpha}(M;E)$. From this perspective, the functions $u_{\delta,\lambda}$ with $\delta=\lambda^{-\gamma}$ are probing the principal symbol of $P$ at the front face. Since $P$ is a smooth coefficient operator on $M$, this is the same as probing its principal symbol over $\pa M$.}  Dividing by $\lambda^m$ and letting $\lambda\to\infty$ gives
  \[
    c \leq C\|\sigmab^m(P)(0,\cdot;\xi,\eta)\psi(\cdot)\|_{L^2(M;F)}.
  \]
  Since $\psi$ (with norm $1$ and support in the coordinate chart on $\pa M$) is arbitrary, this implies the injectivity of $\sigmab^m(P)(0,y,\xi,\eta)$ for all $y$ in the coordinate chart.
\end{proof}

\begin{thm}[Arbitrary index sets for elements of the kernel]
\label{ThmAsIndetPhg}
  Let $P\in\Diffb^m(M;E,F)$ be underdetermined-elliptic. Then for any index set $\cE\subset\C\times\N_0$, there exist an index set $\cE'\supset\cE$ and a solution $u\in\cA_\phg^{\cE'}(M;E)$ of $P u=0$ with the property that $u\notin\cA_\phg^\cF(M;E)$ for all index sets $\cF$ which do not contain $\cE$. In this sense, polyhomogeneous elements of $\ker P$ can have arbitrary index sets.
\end{thm}
\begin{proof}
  Let $(z,k)\in\cE$. Since the principal symbol of $P$, and thus of $N(P,z)$ for all $z\in\C$, is surjective but not injective, we have $\ord(P,z)=\infty$ (as remarked after Lemma~\ref{LemmaAsPair}). This allows us to pick $u'\in F_J(P,z)$, $J\geq\max(k,\ord(P^*,\bar z+i w))$; then $f=-P(\chi u')$ is polyhomogeneous, lies in $\cA^{\Re z+1-\eps}(M;F)$ for all $\eps>0$, and integrates to $0$ against all $u^*\in\bigcap_{\eps>0}\cA^{-\Re z-w+\eps}(M;F)$ with $P^*u^*=0$ (as follows from an integration by parts). Therefore, Theorem~\ref{ThmAs} applies and produces a polyhomogeneous $u''$ with $P u''=f$ with the property that the largest power of $\log\rho$ in the term $\rho^{i z}(\log\rho)^\ell$ of the expansion of $u''$ is $\ell<\ord(P^*,\bar z+i w)\leq J$; hence $u_z:=\chi u'+u''\in\bigcap_{\eps>0}\cA^{\Re z-\eps}(M;E)\cap\ker P$ is polyhomogeneous, and its index set contains $(z,k)$. We may then take $u=\sum_{(z,k)\in\cE}\eps_z u_z$ where $\eps_z$ tends to $0$ sufficiently fast as $|z|\to\infty$ so as to ensure convergence of the sum in a space of polyhomogeneous conormal functions, and to ensure the absence of any cancellations among the terms in the expansions of the $u_z$ which would reduce the size of the index set of $u$.
\end{proof}

\section{Applications}
\label{SA}

\subsection{Geometric operators on asymptotically Euclidean spaces}
\label{SsAG}

Let $n\in\N$, $n\geq 2$. Let $M^\circ$ be a smooth connected $n$-dimensional manifold without boundary; we assume that there exists a compact set $K\subset M^\circ$ so that $M^\circ\setminus K\cong\R^n\setminus B(0,R)$ with $R>0$. Denote by
\[
  g\in\CI(M^\circ;S^2 T^*M^\circ)
\]
an \emph{asymptotically flat metric}: by this we mean that $g$ is Riemannian, and on $\R^n\setminus B(0,R)$ the metric coefficients $g_{i j}$ (in the standard coordinates $x$ on $\R^n$) are equal to $\delta_{i j}$ (Euclidean metric) plus error terms which are smooth functions of $\rho=|x|^{-1}\geq 0$ and $\omega=\frac{x}{|x|}\in\Sph^{n-1}$ which vanish at $\rho=0$.\footnote{The assumptions on $g_{i j}-\delta_{i j}$ can be weakened. For mapping properties on Schwartz spaces, conormality (i.e.\ infinite b-regularity) on $M$ (defined below) suffices, whereas for polyhomogeneous results one only needs to assume the polyhomogeneity of $g_{i j}-\delta_{i j}$ at $\rho=0$, though the index sets of the solutions would get additional contributions from the index set of $g_{i j}-\delta_{i j}$.}

Define $M$ as the radial compactification of $M^\circ$ at infinity, i.e.\ $M=(M^\circ\cup([0,R^{-1})_\rho\times\Sph^{n-1}))/\sim$ where we identify $x=(x^1,\ldots,x^n)\in\R^n\setminus B(0,R)$ with $(\rho,\omega)=(|x|^{-1},\frac{x}{|x|})$. Thus, $g_{i j}\in\CI(M\setminus K)$ and $g_{i j}-\delta_{i j}\in\rho\CI(M\setminus K)$, where we write $\rho\in\CI(M)$ for a function which is positive on $M^\circ$ and equal to $|x|^{-1}$ near $\pa M$. In the case $M^\circ=\R^n$, we then have $\CIdot(M)=\sS(\R^n)$ (Schwartz space) and $\CmI(M)=\sS'(\R^n)$ (tempered distributions). 

Let $\Tsc M\to M$ be the smooth vector bundle which equals $T M^\circ$ over $M^\circ$, and which in the collar neighborhood $\cU=[0,R^{-1})_\rho\times\Sph^{n-1}$ of $\pa M$ is trivialized by $\Tsc_\cU M=\cU\times\R^n$, where $(x,v)\in(\cU\setminus\pa M)\times\R^n$ is identified with $v^j\pa_{x^j}\in T_x M^\circ$. That is, $\{\pa_{x^j}\colon 1\leq j\leq n\}$ extends from $M^\circ\setminus K$ to a smooth frame of $\Tsc M$ over $M\setminus K$. (This is the \emph{scattering tangent bundle} in the terminology of \cite{MelroseEuclideanSpectralTheory}.) Write $\Tsc^*M$ for the dual bundle of $\Tsc M$; thus, $\{\dd x^j\colon 1\leq j\leq n\}$ is a smooth frame over $M\setminus K$. Then $g\in\CI(M;S^2\,\Tsc^*M)$. The metric volume density $|\dd g|$ has weight $w=-n$.

\begin{lemma}[Connection]
\label{LemmaAGNabla}
  The Levi-Civita connection $\nabla$ of $g$ satisfies
  \[
    \nabla\in\rho\Diffb^1(M;\Tsc M,\Tsc^*M\otimes\Tsc M).
  \]
  The normal operator $N(\nabla)$ is the Levi-Civita connection (expressed in inverse polar coordinates) on Euclidean space without the origin. Moreover, $\dd\in\rho\Diffb^1(M;\ul\R,\Tsc^*M)$ where $\ul\R=M\times\R$ is the trivial bundle.
\end{lemma}
\begin{proof}
  Over $M^\circ$, this merely states that $\nabla$ is a smooth coefficient differential operator. Near $\pa M$, we work in local coordinates. In the region where $x^1\geq\delta|x^k|$, $k=2,\ldots,n$, $\delta>0$, smooth coordinates are $\tilde\rho=\frac{1}{x^1}\geq 0$ and $y^k=\frac{x^k}{x^1}$. Thus $g_{i j}=\delta_{i j}+\tilde\rho\tilde g_{i j}(\tilde\rho,y)$ where $\tilde g_{i j}$ is smooth. Since $\pa_{x^\ell}\tilde\rho=-\delta_{\ell 1}\tilde\rho^2\in\tilde\rho^2\CI$ and $\pa_{x^\ell}y^k\in\tilde\rho\CI$, we have $\pa_{x^\ell}g_{i j}\in\tilde\rho^2\CI$, and therefore the Christoffel symbols $\Gamma_{k\ell}^j$ of $g$ (in $x$-coordinates) lie in $\tilde\rho^2\CI$. Since
  \begin{equation}
  \label{ItAGDeriv}
    \pa_{x^1}=-\tilde\rho(\tilde\rho\pa_{\tilde\rho}+y^k\pa_{y^k}),\qquad
    \pa_{x^k}=\tilde\rho\pa_{y^k}
  \end{equation}
  lie in $\tilde\rho\Diffb^1$, we conclude that
  \[
    \nabla(v^j\pa_{x^j}) = (\pa_{x^\ell}v^j+\Gamma_{k\ell}^j v_k)\,\dd x^\ell\otimes\pa_{x^j}
  \]
  is of the stated class. The membership for $\dd$ follows from $\pa_{x^j}\in\rho\Diffb^1$.
\end{proof}

Therefore, the results in~\S\S\ref{Sb}--\ref{SAs} apply to geometric operators on $(M^\circ,g)$ under appropriate assumptions on their principal symbols. The following is a small selection of simple such results.\footnote{Some parts of these results are quite elementary; we state them only as illustrations of our general results.}

\begin{thm}[Divergence on 1-forms]
\label{ThmAGdiv}
  Let $g$ be an asymptotically flat metric on $M$. Denote by $\delta_g$ the codifferential, and consider the equation
  \[
    \delta_g\omega=u.
  \]
  \begin{enumerate}
  \item\label{ItAGdivSchw} Given $u\in\CIdot(M)$, there exists a solution $\omega\in\CIdot(M;\Tsc^*M)$ (so the components $\omega_i=\omega(\pa_{x^i})$ are rapidly decaying as $|x|\to\infty$) if and only if $\int_M u\,\dd g=0$. If this condition is violated, then there still exists a solution $\omega\in\rho^{n-1}\CI(M;\Tsc^*M)$.
  \item\label{ItAGdivPhg} More generally, suppose $u\in\cA_\phg^\cF(M)$ where $\cF\subset\C\times\N_0$ is an index set; let $\alpha_\cF=\min_{(z,0)\in\cF}\Re z$. If $\alpha_\cF>n$, then there exists a solution $\omega\in\cA_\phg^{\cF-1}(M;\Tsc^*M)$ when $\int_M u\,\dd g=0$, and $\omega\in\cA_\phg^{\cF-1}(M;\Tsc^*M)+\rho^{n-1}\CI(M;\Tsc^*M)$ otherwise; if $\alpha_\cF\leq n$, then there exists a solution $\omega\in\cA_\phg^{\cE-1}(M;\Tsc^*M)+\rho^{n-1}\CI(M;\Tsc^*M)$ where $\cE=\{(z+j,k+\ell)\colon (z,k)\in\cF,\ j\in\N_0,\ \ell\leq\ell(z,j)\}$ where $\ell(z,j)=1$ when $z\in n-\N_0$ and $z+j\in n+\N_0$, and $\ell(z,j)=0$ otherwise.
  \end{enumerate}
\end{thm}
\begin{proof}
  The symbol of $(\delta_g)^*=\dd$ near $\pa M$ is injective, as follows from the expressions~\eqref{ItAGDeriv}. Moreover, if $u^*\in\CmI(M)$ and $\dd u^*=0$, then $u^*$ is constant. The first claim of part~\eqref{ItAGdivSchw} then follows by applying Corollary~\ref{CorbSchwartz} to the operator $\rho^{-1}\delta_g\in\Diffb^1(M;\Tsc^*M,\ul\R)$. Moreover, since $\injSpecb(\rho^{-1}\dd)=\{(0,0)\}$, we have $\injSpecb(\dd\circ\rho^{-1})=\{(1,0)\}$ and thus $\surjSpecb(\rho^{-1}\delta_g)=\{(n-1,0)\}$. Therefore, Theorem~\ref{ThmAs} applies to $\dd\circ\rho^{-1}$ (with $\alpha_\cF=+\infty$ and $\alpha_{\rm coker}=\alpha_0=n-1$); note that $\cE(\rho^{-1}\delta_g,n-1)=\{(n-1+j,0)\colon j\in\N_0\}$ and $\cA_\phg^{\cE(\rho^{-1}\delta_g,n-1)}(M;\Tsc^* M)=\rho^{n-1}\CI(M;\Tsc^* M)$. Part~\eqref{ItAGdivPhg} follows similarly from Theorem~\ref{ThmAs}.
\end{proof}

\begin{thm}[Divergence on symmetric 2-tensors]
\label{ThmAGdiv2}
  Let $\delta_g$ be the (negative) divergence on symmetric 2-tensors on $(M^\circ,g)$, and consider the equation
  \[
    \delta_g h=\omega.
  \]
  Let $\omega\in\CIdot(M;\Tsc^*M)$. If $(M^\circ,g)$ does not admit any nontrivial Killing vector fields, or more generally if $\omega$ is orthogonal to all Killing vector fields, then there exists a solution $h\in\CIdot(M;S^2\,\Tsc^*M)$. If this condition is violated, then there exists a solution $h\in\rho^{n-1}\CI(M;S^2\,\Tsc^*M)+\rho^n(\log\rho)\CI(M;S^2\,\Tsc^*M)$.
\end{thm}

We leave the statement of a polyhomogeneous version of this result to the reader.

\begin{proof}[Proof of Theorem~\usref{ThmAGdiv2}]
  The adjoint $\delta_g^*$ (symmetric gradient) of $\delta_g$ has injective principal symbol; its (tempered) distributional kernel is given by the space of all Killing 1-forms. Thus the first claim follows again from Corollary~\ref{CorbSchwartz}. The normal operator of $\delta_g^*$ is the Euclidean symmetric gradient $\delta^*$, with kernel spanned by the generators of translations ($\dd x^j$) and rotations ($x^i\,\dd x^j-x^j\,\dd x^i$). Thus, $\injSpecb(\rho^{-1}\delta_g^*)=\{(0,0),(-1,0)\}$ and thus $\injSpecb(\delta_g^*\circ\rho^{-1})=\{(0,0),(1,0)\}$, and therefore $\surjSpecb(\rho^{-1}\delta_g)=\{(n-1,0),(n,0)\}$.
\end{proof}

\begin{rmk}[Divergence on Euclidean space, I]
\label{RmkAGdivEucl1}
  If $g$ is the Euclidean metric, or more generally if $g$ is spherically symmetric to leading and subleading order at $\pa M$, then one can avoid a logarithmic term in Theorem~\ref{ThmAGdiv2}, i.e.\ one can find $h\in\rho^{n-1}\CI(M;S^2\,\Tsc^*M)$. Indeed, in the proof of Proposition~\ref{PropAsSchw}, the expansion term corresponding to $(n-1,0)\in\surjSpecb(\rho^{-1}\delta_g)$ can, via an averaging procedure, be taken to be of scalar type $1$ (since the kernel of $N(\rho^{-1}\delta_g^*,0)$ consists of scalar type $1$ 1-forms). The spherical symmetry assumption then ensures that the resulting error subleading term is still of scalar type $1$, and thus orthogonal to the space of vector type $1$ 1-forms which the kernel of $N(\rho^{-1}\delta_g^*,-1)$ is a subspace of, and hence can be solved away without the introduction of a logarithmic term.
\end{rmk}

\begin{rmk}[Divergence on Euclidean space, II]
\label{RmkAGdivEucl2}
  On exact Euclidean space, \cite[Theorem~4]{MaoTaoConstraints} (building on \cite[Proposition~4.1]{OhTataruYangMills3}) produces a family of right inverses of the divergence operators on 1-forms and symmetric 2-tensors which depends on a choice of smooth function on $\Sph^{n-1}$ which encodes localization to conic regions. Acting on rapidly vanishing right hand sides, they all produce different solutions of the divergence equation than the one constructed here. It is an interesting problem to generalize our arguments so as to ensure localization properties of solutions, e.g.\ so that in the context of Theorem~\ref{ThmAGdiv2} the solution $h$ is supported in the same conic region as $\omega$.
\end{rmk}

\subsection{Sharp asymptotics for initial data gluing}
\label{SsAI}

\emph{This section is not self-contained}: we shall only indicate how the results of the present paper can be used to obtain more precise asymptotics in the gluing method for the constraint equations in general relativity introduced by the author in \cite{HintzGlueID}. Recall that the constraint equations are $P(\gamma,k)=0$ where $\gamma$ is a Riemannian metric and $k$ is a symmetric 2-tensor on an $n$-dimensional manifold, and
\begin{align*}
  &P(\gamma,k) := (P_1(\gamma,k), P_2(\gamma,k)), \\
  &\qquad P_1(\gamma,k):=R_\gamma-|k|_\gamma^2+(\tr_\gamma k)^2, \qquad
  P_2(\gamma,k):=\delta_\gamma k+\dd(\tr_\gamma k),
\end{align*}
with $R_\gamma$ denoting the scalar curvature of $\gamma$. We first revisit a key result in the construction of a formal solution of the gluing problem:

\begin{prop}[Linearized constraints map around asymptotically flat initial data]
\label{PropAIAF}
  (Cf.\ \cite[Proposition~4.10(2)]{HintzGlueID}.) Let $\hat K\subset\R^n$, $n\geq 3$, be compact (possibly empty) and contained in the open Euclidean ball of radius $\hat R_0>0$ around $0$. Set $\rho_\circ=\la\hat x\ra^{-1}$ on $\R^n$ with standard coordinates denoted $\hat x$. Fix the index set $\cE=\{(n-2+j,0)\colon j\in\N_0\}$, and put
  \[
    \hat\cS = \{ (n-2,0) \} \cup \{ (n-1+j,k) \colon j\in\N_0,\ k=0,1 \}.
  \]
  Let $(\hat\gamma,\hat k)$ be $\cE$-asymptotically flat initial data; that is, $\hat\gamma\in\rho_\circ^{n-2}\CI(\ol{\R^n}\setminus\hat K^\circ;S^2\,\Tsc^*\ol{\R^n})$ and $\hat k\in\rho_\circ^{n-1}\CI(\ol{\R^n}\setminus\hat K^\circ;S^2\,\Tsc^*\ol{\R^n})$ solve the constraint equations $P(\hat\gamma,\hat k)=0$. Denote by $L_{\hat\gamma,\hat k}$ the linearization of $P(-,-)$ at $(\hat\gamma,\hat k)$. If $\cG\subset\C\times\N_0$ is an index set, and if $\hat f\in\cA_\phg^{\cG+2}(\ol{\R^n})$, $\hat j\in\cA_\phg^{\cG+2}(\ol{\R^n};\Tsc^*\ol{\R^n})$ vanish in a neighborhood of $|\hat x|\leq\hat R_0$, then there exist\footnote{We use here the extended union of index sets, $\cE\extcup\cF:=\cE\cup\cF\cup\{(z,k+\ell+1)\colon(z,k)\in\cE,\ (z,\ell)\in\cF\}$.}
  \begin{equation}
  \label{EqAIAF}
    \hat h\in\cA_\phg^{\cG\extcup\hat\cS}(\ol{\R^n};S^2\,\Tsc^*\ol{\R^n}),\quad
    \hat q\in\cA_\phg^{\cG\extcup\hat\cS+1}(\ol{\R^n};S^2\,\Tsc^*\ol{\R^n})
  \end{equation}
  so that $L_{\hat\gamma,\hat k}(\hat h,\hat q)=(\hat f,\hat j)$, and so that $\hat h$ and $\hat q$ vanish in a neighborhood of $|\hat x|\leq\hat R_0$.
\end{prop}

The index set $\cE$ is particularly natural since the initial data of (higher-dimensional) Schwarzschild and Kerr metrics have the regularity and decay required for $(\hat\gamma,\hat k)$, as demonstrated in \cite[Lemma~6.1]{HintzGlueID} in the case $n=3$. Of course, one can (but we shall not) similarly sharpen \cite[Proposition~4.10(2)]{HintzGlueID} (and thus also \cite[Theorem~5.2]{HintzGlueID}, as discussed below) for asymptotically flat data with general index sets $\cE$.

\begin{proof}[Proof of Proposition~\usref{PropAIAF}]
  Setting $\hat w=\diag(\rho_\circ^{-2},\rho_\circ^{-1})$, the operator $L:=L_{\hat\gamma,\hat k}\hat w$ is a totally characteristic operator whose principal symbol, in the Douglis--Nirenberg sense as an element of $(\Diffb^{t_j+s_i})_{i,j=1,2}$, $t_1=t_2=0$, $s_1=2$, $s_2=1$, is surjective; see \cite[Lemmas~4.2 and 4.7]{HintzGlueID}. By \cite[Lemma~4.6]{HintzGlueID}, its normal operator is block-diagonal and given by $\diag(\hat{\ubar L}_1,\hat{\ubar L}_2)\hat{\ubar w}$ where $\hat{\ubar L}_1=\Delta\tr+\delta\delta$ and $\hat{\ubar L}_2=\delta+\dd\tr$ are geometric operators on Euclidean space $\R^n_{\ubar x}\setminus\{0\}$ and $\hat{\ubar w}=\diag(|\ubar x|^2,|\ubar x|^1)$. Then \cite[Lemma~4.8]{HintzGlueID} computes the injective boundary spectrum of $L^*$ to be $\{(-1,0),(0,0)\}$; since the volume density of $\hat\gamma$ has weight $-n$, this gives $\surjSpecb(L)=\{(n,0),(n+1,0)\}$.

  If it were not for the support requirements on $\hat h,\hat q$ (and the possible presence of $\hat K$), an application of Theorem~\ref{ThmAs} to the equation $L(h',q')=(\hat f,\hat j)$ would finish the proof upon setting $\hat h=\rho_\circ^{-2}h'$, $\hat q=\rho_\circ^{-1}q'$. (In fact, the asymptotic behavior produced by Theorem~\ref{ThmAs} is slightly stronger still than~\eqref{EqAIAF} as far as the powers of $\log\rho_\circ$ in subleading terms are concerned.) We briefly sketch how to account for the support requirements. One first uses Proposition~\ref{PropAs} to produce, locally near $\pa\ol{\R^n}$, a formal solution of the linearized constraint equations. To solve away the remaining Schwartz errors, one works on suitable weighted b-00-Sobolev spaces on the closure in $\ol{\R^n}$ of $\{|\hat x|\geq\hat R_0+\eta\}$ for some small $\eta>0$, on which the cokernel of $L$ is finite-dimensional; upon eliminating the cokernel as in the proof of Proposition~\ref{PropAsSchw}, one can apply a variant of Corollary~\ref{CorbSchwartz} on such Sobolev spaces (with rapid vanishing at $\pa\ol{\R^n}$, with a fixed weight at $|\hat x|=\hat R_0+\eta$, and with infinite b-00-regularity) to conclude.
\end{proof}

In a similar vein, one can sharpen \cite[Proposition~4.14(2)]{HintzGlueID} so as to ensure that the index set of the solution of the linearized constraint equations on the punctured manifold $X_\circ=[X;\{\fp\}]$ (in the notation of the reference) is log-smooth at the conic point (the lift of $\{\fp\}$) when the right hand side is.

\begin{cor}[Log-smoothness of the glued data]
\label{CorAILogCI}
  We use the notation and terminology of \cite[Theorem~5.2]{HintzGlueID} and make the same assumptions, with index set $\cE$ as in Proposition~\usref{PropAIAF}; that is, the boundary data are a smooth solution $(\gamma,k)$ of the constraint equations on a smooth $n$-dimensional manifold $X$ subject to a local genericity condition (absence of KIDs) near a point $\fp\in X$, and an $\cE$-asymptotically flat solution $(\hat\gamma,\hat k)$ as in Proposition~\usref{PropAIAF}. Then there exist index sets $\hat\cE_\sharp\subset(1+\N_0)\times\N_0$ and $\cE_\sharp\subset(n-2+\N_0)\times\N_0$ for which the conclusions of \cite[Theorem~5.2]{HintzGlueID} hold. That is, there exists a \emph{log-smooth} total family $(\wt\gamma,\wt k)=(\wt\gamma_\eps,\wt k_\eps)_{\eps\in(0,\eps_\sharp)}$ on the total gluing space (see \cite[Definition~3.1]{HintzGlueID}), with boundary data $(\gamma,k)$ and $(\hat\gamma,\hat k)$, which solves the constraint equations.
\end{cor}
\begin{proof}
  We follow \cite[\S5.1]{HintzGlueID} to construct a formal solution of the constraint equations on the total gluing space, except we now use Proposition~\ref{PropAIAF} instead of \cite[Propositions~4.10(2)]{HintzGlueID}, and the similarly sharpened version of \cite[Proposition~4.14(2)]{HintzGlueID}; this ensures the log-smoothness. (We leave the problem of obtaining bounds on the exponents of the logarithms to the interested reader.) The correction of the formal solution to a true solution is done exactly as in \cite[\S5.2]{HintzGlueID}.
\end{proof}

\appendix
\section{Characterization of the closed range property}
\label{SCl}

For completeness, we prove here a classical functional analytic result about bounded linear operators between Fr\'echet spaces. References include \cite[\S37]{TrevesTVS}, \cite[Chap.~IV, 6.4 and 7.7]{SchaeferWolffTVS}, and \cite[21.9 and 22.7]{KelleyNamiokaLinearTopologicalSpaces}. For the convenience of the reader, we give self-contained proofs here, following~\cite{HormanderFA}.

Let $E,F$ be Fr\'echet spaces, and let $T\colon E\to F$ be a continuous linear operator. Denote by $p_1\leq p_2\leq\cdots$ a family of seminorms on $E$ so that $\{x\in E\colon p_j(x)\leq 1\}$ is a decreasing fundamental system of neighborhoods of $0\in E$; let $q_1\leq q_2\leq\cdots$ denote an analogous family of seminorms on $F$. For a continuous seminorm $p$ on $E$, define by
\[
  U_p = \{ \lambda\in E^* \colon |\lambda(x)|\leq p(x)\ \forall\,x\in E\} \subset E^*
\]
the polar of $\{x\in E\colon p(x)\leq 1\}$.

\begin{thm}[Characterization of the closed range property]
\label{ThmCl}
  The range $\ran T=T(E)\subset F$ is closed if and only if the set $\ran T^*\cap U_{p_n}\subset E^*$ is weak* closed for all $n\in\N$. In this case, $\ran T=\ann\ker T^*$ and $\ann\ker T=\ran T^*$.
\end{thm}

This is an immediate consequence of the following two results:

\begin{prop}[Closed range and adjoints]
\label{PropClRan}
  $T(E)\subset F$ is closed if and only if $\ran T^*\subset E^*$ is weak* closed. In this case $\ran T=\ann\ker T^*$ and $\ann\ker T=\ran T^*$.
\end{prop}
\begin{proof}
  Suppose $\ran T$ is closed. If $\mu=T^*\lambda$ for $\lambda\in F^*$, then for $x\in\ker T$ we have $\mu(x)=(T^*\lambda)(x)=\lambda(T x)=0$. If on the other hand $\mu\in\ann\ker T$, then $\mu$ induces a bounded linear map $[\mu]\colon E/\ker T\to\C$. Set $\lambda_0\colon\ran T\to\C$, $\lambda_0(T x)=\mu(x)=[\mu]([x])$, where $[x]=x+\ker T\in E/\ker T$; this is well-defined since $T x=0$ implies $[x]=0$. But by the Open Mapping Theorem, the map $E/\ker T\to\ran T$ induced by $T$ is an isomorphism, and therefore its inverse $T x\mapsto[x]$ is continuous. Therefore, $\lambda_0$ is continuous, and by the Hahn--Banach theorem has a continuous extension $\lambda\in F^*$. By construction, $(T^*\lambda)(x)=\lambda(T x)=\lambda_0(T x)=\mu(x)$ for all $x\in E$, so $T^*\lambda=\mu$.

  Conversely, suppose $\ran T^*$ is weak* closed. Then for $V=\ann\ran T^*=\ker T\subset E$, we have $\ann V=\ran T^*$. Factor now $T=T_1 \pi$, where $T_1\colon E/V\to F$ is the (injective) map induced by $T$, and $\pi\colon E\to E/V$ is the projection. Then $T^*=\pi^* T_1^*$, with $\pi^*\colon(E/V)^*\to E^*$ being injective with range $\ann V=\ran T^*$. Therefore, $T_1^*$ is surjective. Replacing $T$, $E$ by $T_1$, $E/V$, we may thus assume that $T$ is injective and $T^*$ is surjective.

  We first show that if $(x_j)_{j\in\N}$ is a sequence in $E$ with $T x_j\to 0$, then $x_j\to 0$. For all $n\in\N$, we have $q_n(T x_j)\to 0$, and thus we may inductively find $j_1<j_2<\ldots$ so that $q_n(T x_j)<\frac{1}{n}$ when $j\geq j_n$ (and thus $q_m(T x_j)<\frac{1}{n}$ for $m\leq n$). Let $\eps_j=\frac{1}{n}$ for $j_n\leq j<j_{n+1}$, which thus satisfies $\eps_j\to 0$; then $T x_j/\eps_j$ is bounded. Given any $\mu\in F^*$, the sequence $\mu(T x_j/\eps_j)=(T^*\mu)(x_j/\eps_j)$ is then bounded; since $T^*$ is surjective, this implies that $\lambda(x_j/\eps_j)$ is bounded for all $\lambda\in E^*$. We claim that this implies that $x_j/\eps_j$ is bounded in $E$. To see this, note that if $p$ is any continuous seminorm on $E$, then the subset $E^*_p\subset E^*$ of $\lambda\in E^*$ which are continuous with respect to $p$ is a Banach space with norm $\|\lambda\|_p=\sup_{x\in E,\ p(x)=1} |\lambda(x)|$. But $x_j/\eps_j\colon E^*_p\to\C$, $\lambda\mapsto\lambda(x_j/\eps_j)$, is continuous in view of $|\lambda(x_j/\eps_j)|\leq\|\lambda\|_p p(x_j/\eps_j)$, and it is pointwise bounded; therefore, it is uniformly bounded by the Banach--Steinhaus theorem. This implies that $p(x_j/\eps_j)$ is bounded, and therefore $x_j\to 0$ since $p$ was arbitrary.

  Finally, if $T x_j\to y\in F$, then via a diagonal argument we can find a subsequence, again denoted $x_j$, so that $q_n(T x_{j+1}-T x_j)<2^{-2 j}$ for all $j\geq n$. Thus $T(2^j(x_{j+1}-x_j))\to 0$, which implies that $2^j(x_{j+1}-x_j)\to 0$, and therefore $x_1+\sum_{j=1}^\infty 2^{-j}(2^j(x_{j+1}-x_j))$ converges in $E$ to a limit $x$ which satisfies $T x=y$. The proof is complete.
\end{proof}

\begin{prop}[Weak* closed subspaces]
\label{PropClW}
  A linear subspace $W\subset E^*$ is weak* closed if and only if $W\cap U_{p_n}$ is weak* closed for all $n\in\N$.
\end{prop}
\begin{proof}
  Since $U_{p_n}$ is weak* compact by the Banach--Alaoglu theorem, one implication is obvious. Conversely, suppose $W\cap U_{p_n}$ is weak* closed, thus weak* compact, for all $n\in\N$. Let $\lambda_0\notin W$. We first claim that there exists a sequence $(x_j)_{j\in\N}$ with $x_j\to 0$ so that $\lambda\in E^*$,
  \begin{equation}
  \label{EqClWDiff}
    |\lambda(x_j)-\lambda_0(x_j)|\leq 1
  \end{equation}
  for all $j$ implies that $\lambda\notin W$. Let $N$ be such that $\lambda_0\in C U_{p_N}$ for some $C>0$. By relabeling $C p_N,C p_{N+1},\ldots$ as $p_1,p_2,\ldots$, we may assume that $\lambda_0\in U_{p_1}$ (and thus $\lambda_0\in U_{p_n}$ for all $n$). Define $W_n:=W\cap(U_{p_n}+\lambda_0)\subset W\cap U_{2 p_n}$, which is weak* compact. Suppose that we have found $x_1,\ldots,x_k\in E$ so that~\eqref{EqClWDiff} for $1\leq j\leq k$ implies $\lambda\notin W_n$; for $n=1$, we can find such elements of $E$ since $W_1$ is weak* closed and $\lambda_0\notin W_1$. Note then that
  \[
    W_{n+1} \cap \{ \lambda\in E^* \colon |\lambda(x_j)-\lambda_0(x_j)|\leq 1,\ j=1,\ldots,k \}
  \]
  is a weak* compact subset of $W_{n+1}$ that is disjoint from $W_n$; it has empty intersection with $\bigcap_{x\in E,\ p_n(x)\leq 1} C_x=U_{p_n}+\lambda_0$ where we define the weak* closed set $C_x=\{\lambda\in E^*\colon|\lambda(x)-\lambda_0(x)|\leq 1\}$. By the finite intersection property, there thus exist finitely many $x_{k+1},\ldots,x_{k+l}\in E$ with $p_n(x_j)\leq 1$, $k+1\leq j\leq k+l$, with the desired property.

  Next, let $c_0$ denote the space of complex-valued null sequences, with the supremum norm $\|\cdot\|_\infty$. Define the map
  \[
    \phi \colon E^* \to c_0,\qquad \lambda \mapsto (\lambda(x_j))_{j\in\N}.
  \]
  Then $\|\phi(\cdot)\|_\infty$ is a continuous seminorm on $E^*$, and $\|\phi(\lambda-\lambda_0)\|_\infty>1$ for all $\lambda\in W$. By the Hahn--Banach theorem, there exists a linear map $f\colon E^*\to\C$ so that
  \begin{equation}
  \label{EqClWf}
    f|_W=0,\qquad
    f(\lambda_0)=1,
  \end{equation}
  and $|f(\lambda)|\leq C\|\phi(\lambda)\|_\infty$. Applying the Hahn--Banach theorem again, we can thus extend the map $c_0\ni\phi(\lambda)\mapsto f(\lambda)\in\C$ to a continuous linear functional $a\in(c_0)^*=\ell^1$. That is, $a=(a_n)_{n\in\N}$ is absolutely summable, and
  \[
    f(\lambda) = \sum_{n=1}^\infty a_n \phi(\lambda)_n = \lambda(x),\qquad x := \sum_{n=1}^\infty a_n x_n\in E.
  \]
  In view of~\eqref{EqClWf}, this implies that $\lambda_0$ has a weak* neighborhood disjoint from $W$, finishing the proof.
\end{proof}

\bibliographystyle{alphaurl}


\begin{thebibliography}{BEM76}

\bibitem[BE69]{BergerEbinTensors}
Marcel Berger and David~G. Ebin.
\newblock Some decompositions of the space of symmetric tensors on a
  {R}iemannian manifold.
\newblock {\em Journal of Differential Geometry}, 3(3-4):379--392, 1969.

\bibitem[BEM76]{BourguignonEbinMarsdenPsdoSurj}
Jean-Pierre Bourguignon, David~G. Ebin, and Jerrold~E. Marsden.
\newblock Sur le noyau des op{\'e}rateurs pseudo-diff{\'e}rentiels {\`a}
  symbole surjectif et non injectif.
\newblock {\em Comptes rendus hebdomadaires des s{\'e}ances de l'Acad{\'e}mie
  des Sciences. A, Sciences math{\'e}matiques, B, Sciences physiques},
  282:867--870, 1976.

\bibitem[Cho]{ChodoshLinStab}
Otis Chodosh.
\newblock Notes on linearization stability.
\newblock {\em Notes for Math 394 at Stanford University}.
\newblock URL: \url{http://web.stanford.edu/~ochodosh/LinStabNOTES.pdf}.

\bibitem[Cor00]{CorvinoScalar}
Justin Corvino.
\newblock Scalar curvature deformation and a gluing construction for the
  {E}instein constraint equations.
\newblock {\em Comm. Math. Phys.}, 214(1):137--189, 2000.

\bibitem[Del12]{DelayCompact}
Erwann Delay.
\newblock {S}mooth compactly supported solutions of some underdetermined
  elliptic {PDE}, with gluing applications.
\newblock {\em Communications in Partial Differential Equations},
  37(10):1689--1716, 2012.

\bibitem[DZ19]{DyatlovZworskiBook}
Semyon Dyatlov and Maciej Zworski.
\newblock {\em Mathematical theory of scattering resonances}, volume 200 of
  {\em Graduate Studies in Mathematics}.
\newblock American Mathematical Society, 2019.

\bibitem[Gri01]{GrieserBasics}
Daniel Grieser.
\newblock Basics of the b-calculus.
\newblock In Juan~B. Gil, Daniel Grieser, and Matthias Lesch, editors, {\em
  Approaches to Singular Analysis: A Volume of Advances in Partial Differential
  Equations}, pages 30--84. Birkh{\"a}user Basel, Basel, 2001.
\newblock \href {https://doi.org/10.1007/978-3-0348-8253-8_2}
  {\path{doi:10.1007/978-3-0348-8253-8_2}}.

\bibitem[Hin24]{HintzGlueID}
Peter Hintz.
\newblock Gluing small black holes into initial data sets.
\newblock {\em Communications in Mathematical Physics}, 405(5):114, Apr 2024.
\newblock \href {https://doi.org/10.1007/s00220-024-04989-6}
  {\path{doi:10.1007/s00220-024-04989-6}}.

\bibitem[H{\"o}r]{HormanderFA}
Lars H{\"o}rmander.
\newblock Functional {A}nalysis.
\newblock {\em Unpublished lecture notes}.

\bibitem[KN76]{KelleyNamiokaLinearTopologicalSpaces}
John~L. Kelley and Isaac Namioka.
\newblock {\em Linear topological spaces}.
\newblock Graduate Texts in Mathematics, No. 36. Springer-Verlag, New
  York-Heidelberg, 1976.
\newblock With the collaboration of W. F. Donoghue, Jr., Kenneth R. Lucas, B.
  J. Pettis, Ebbe Thue Poulsen, G. Baley Price, Wendy Robertson, W. R. Scott,
  and Kennan T. Smith, Second corrected printing.

\bibitem[Mel93]{MelroseAPS}
Richard~B. Melrose.
\newblock {\em The {A}tiyah-{P}atodi-{S}inger index theorem}, volume~4 of {\em
  Research Notes in Mathematics}.
\newblock A K Peters, Ltd., Wellesley, MA, 1993.
\newblock \href {https://doi.org/10.1016/0377-0257(93)80040-i}
  {\path{doi:10.1016/0377-0257(93)80040-i}}.

\bibitem[Mel94]{MelroseEuclideanSpectralTheory}
Richard~B. Melrose.
\newblock Spectral and scattering theory for the {L}aplacian on asymptotically
  {E}uclidian spaces.
\newblock In {\em Spectral and scattering theory ({S}anda, 1992)}, volume 161
  of {\em Lecture Notes in Pure and Appl. Math.}, pages 85--130. Dekker, New
  York, 1994.

\bibitem[MM83]{MelroseMendozaB}
Richard~B. Melrose and Gerardo Mendoza.
\newblock {\em Elliptic operators of totally characteristic type}.
\newblock Mathematical Sciences Research Institute, 1983.

\bibitem[MT22]{MaoTaoConstraints}
Yuchen Mao and Zhongkai Tao.
\newblock Localized initial data for {E}instein equations.
\newblock {\em Preprint, arXiv:2210.09437}, 2022.

\bibitem[OT19]{OhTataruYangMills3}
Sung-Jin Oh and Daniel Tataru.
\newblock {T}he {H}yperbolic {Y}ang--{M}ills {E}quation for {C}onnections in an
  {A}rbitrary {T}opological {C}lass.
\newblock {\em Communications in Mathematical Physics}, 365(2):685--739, Jan
  2019.
\newblock \href {https://doi.org/10.1007/s00220-018-3205-x}
  {\path{doi:10.1007/s00220-018-3205-x}}.

\bibitem[Shu87]{ShubinSpectralTheory}
Mikhail~A. Shubin.
\newblock {\em Pseudodifferential operators and spectral theory}, volume 200.
\newblock Springer, 1987.

\bibitem[SW99]{SchaeferWolffTVS}
H.~H. Schaefer and M.~P. Wolff.
\newblock {\em Topological vector spaces}, volume~3 of {\em Graduate Texts in
  Mathematics}.
\newblock Springer-Verlag, New York, second edition, 1999.
\newblock \href {https://doi.org/10.1007/978-1-4612-1468-7}
  {\path{doi:10.1007/978-1-4612-1468-7}}.

\bibitem[Tr{\`e}67]{TrevesTVS}
Fran\c{c}ois Tr{\`e}ves.
\newblock {\em Topological vector spaces, distributions and kernels}.
\newblock Academic Press, New York-London, 1967.

\bibitem[Vas21]{VasyLowEnergy}
Andr{\'a}s Vasy.
\newblock {R}esolvent near zero energy on {R}iemannian scattering
  (asymptotically conic) spaces.
\newblock {\em Pure and Applied Analysis}, 3(1):1--74, 2021.
\newblock \href {https://doi.org/10.2140/paa.2021.3.1}
  {\path{doi:10.2140/paa.2021.3.1}}.

\end{thebibliography}

\end{document}